\documentclass{amsart}
\usepackage{amsmath}
\usepackage{amsthm}
\usepackage{amssymb}

\usepackage{enumitem}
\usepackage{mathrsfs}

\newtheorem{theorem}{Theorem}[section]

\newtheorem{alphatheorem}{Theorem}
\renewcommand*{\thealphatheorem}{\Alph{alphatheorem}}

\theoremstyle{definition}

\newtheorem{definition}[theorem]{Definition}
\newtheorem{proposition}[theorem]{Proposition}
\newtheorem{step}{Step}

\newtheorem*{proposition*}{Proposition}
\newtheorem*{observation*}{Observation}
\newtheorem*{claim*}{Claim}

\newtheorem*{lemma*}{Lemma}
\newtheorem{example}[theorem]{Example}
\newtheorem{corollary}[theorem]{Corollary}
\newtheorem{remark}[theorem]{Remark}

\newtheorem*{conjecture*}{Conjecture}

\newtheorem*{convention*}{Convention}

\newtheorem{question}{Question}

\newcommand{\Imm}{\mathrm{Im}}
\theoremstyle{plain}
\newtheorem{lemma}[theorem]{Lemma}

\newcommand{\bra}[1]{ \left( #1 \right) }

\renewcommand{\tilde}{\widetilde}

\newcommand{\abs}[1]{\left|#1\right|}
\newcommand{\fp}[1]{\left\{ #1 \right\}}
\newcommand{\br}[1]{\lbr #1 \rbr}

\newcommand{\ip}[1]{\left[ #1 \right]}

\newcommand{\norm}[1]{\left\lVert #1 \right\rVert}

\renewcommand{\b}{\beta}

\newcommand{\fpa}[1]{\left\lVert #1 \right\rVert}

\newcommand{\e}{\varepsilon}
\renewcommand{\a}{\alpha}
\renewcommand{\b}{\beta}

\usepackage{mathrsfs}
\usepackage[mathscr]{euscript}

\newcommand{\NN}{\mathbb{N}}
\newcommand{\QQ}{\mathbb{Q}}
\newcommand{\Q}{\QQ}

\newcommand{\ZZ}{\mathbb{Z}}
\newcommand{\RR}{\mathbb{R}}
\newcommand{\R}{\RR}
\newcommand{\N}{\NN}
\newcommand{\Z}{\ZZ}
\newcommand{\CC}{\mathbb{C}}

\newcommand{\TT}{\mathbb{T}}

\newcommand{\cB}{\mathcal{B}}
\newcommand{\cF}{\mathcal{F}}
\newcommand{\cX}{\mathscr{X}}
\newcommand{\cG}{\mathcal{G}}

\newcommand{\floor}[1]{\left\lfloor #1 \right\rfloor}

\newcommand{\set}[2]{\left\{ #1 \ \middle| \ #2 \right\} }

\newcommand{\parbreak}[1]{
\begin{center}
***
\end{center}
}

\newcommand{\FS}{\operatorname{FS}}

\usepackage[dvipsnames]{xcolor}

\newcommand{\IP}{$\operatorname{IP}$}
\newcommand{\IPS}{$\mathrm{IPS}$}
\newcommand{\IPP}{$\mathrm{IP}_+$}

\newcommand{\IPd}{$\operatorname{IP}^*$}

\newcommand{\comment}[1]{}

\newcommand{\inter}[1]{\operatorname{int}#1}
\newcommand{\cl}[1]{\operatorname{cl}#1}

\usepackage{soul}

\usepackage{stmaryrd}
\newcommand{\ifbra}[1]{\left\llbracket #1 \right\rrbracket}

\newcommand{\lbr}{\langle\!\langle}

\newcommand{\rbr}{\rangle\!\rangle}

\newcommand{\Haland}{H\r{a}land}
\newcommand{\Malcev}{Mal'cev}

\begin{document}

\address{Department of Mathematics and Computer Science\\Institute of Mathematics\\
Jagiellonian University\\
ul. prof. Stanis\l{}awa \L{}ojasiewicza 6\\
30-348 Krak\'{o}w}
\email{jakub.byszewski@gmail.com}

\author[J. Byszewski \and J. Konieczny]{Jakub Byszewski \and Jakub Konieczny
}
\address{Mathematical Institute \\ 
University of Oxford\\
Andrew Wiles Building \\
Radcliffe Observatory Quarter\\
Woodstock Road\\
Oxford\\
OX2 6GG}
\email{jakub.konieczny@gmail.com}

\title{Sparse generalised polynomials}

\maketitle 

\begin{abstract}
	We investigate generalised polynomials (i.e.\ polynomial-like expressions involving the use of the floor function) which take the value $0$ on all integers except for a set of density $0$. 
	
	Our main result is that the set of integers where a sparse generalised polynomial takes non-zero value cannot contain a translate of an IP set. We also study some explicit constructions, and show that the characteristic functions of the Fibonacci and Tribonacci numbers are given by generalised polynomails. Finally, we show that any sufficiently sparse $\{0,1\}$-valued sequence is given by a generalised polynomial. 
\end{abstract}

\section*{Introduction}

Generalised polynomials are functions given by polynomial-like expressions involving the (possibly iterated) use of the floor function such as $$f(n) = n \lfloor \sqrt{2} n^2 +3\lfloor\sqrt{3} n \rfloor^2 \rfloor .$$ 

Despite some superficial similarities, the class of generalised polynomials differs in fundamental ways from the class of polynomials. For instance, the set of zeroes of a generalised polynomial may be infinite without the polynomial being $0$, which is the case for example for $ f(n) = n/2-\lfloor n/2\rfloor$. As a more striking example, the characteristic function of the Fibonacci numbers can also be expressed as a generalised polynomial (see Example \ref{ex:lin-rec-is-gp}). We  refer to generalised polynomials that take non-zero values on a set of density $0$ as \emph{sparse generalised polynomials}.

In a series of papers \cite{Haland-1993}, \cite{Haland-1994}, \cite{HalandKnuth-1995}, \Haland\ and Knuth studied the equidistribution modulo $1$ of certain generalised polynomials of a specific form. In a more general situation, Bergelson and Leibman \cite{BergelsonLeibman2007} studied the distribution of values of arbitrary generalised polynomials. They proved, among other things, that any generalised polynomial $g \colon \ZZ \to \RR^d$ is well-distributed inside a (parametrised) piecewise polynomial surface (for precise definitions, see \cite{BergelsonLeibman2007}). A very general result on equidistibution modulo $1$ was obtained by Leibman \cite{Leibman-2012}.

Much of the existing theory does not distinguish between two generalised polynomials which differ on a set of density $0$. In particular, most of the results mentioned above are vacuous for sparse generalised polynomials, which from this point of view are indistinguishable from the constant zero function. 

In the present note we propose to study some combinatorial structures which are related to sparse generalised polynomials. Our focus is restricted  mainly to $\{0,1\}$-valued sequences, which can be identified with subsets of $\ZZ$ in a natural way. We will say that a set $E \subset \ZZ$ is generalised polynomial if its characteristic function  is a generalised polynomial; the set $E$ is sparse if its characteristic function is sparse. We ask which combinatorial structures can be found in sparse generalised polynomial sets.

Recall that an \IP\ set is a set containing the set of finite sums of a sequence of positive integers, i.e.\ a set of the form $\set{\sum_{i \in \alpha} n_i }{ \emptyset \neq \alpha \subset \NN,\ \abs{\a} < \infty}$, where $(n_i)_{i \in \NN} \subset \NN$. We are now ready to state our main result.

\begin{alphatheorem}\label{thm:main-A}
	Suppose that $E$ is a sparse generalised polynomial set. Then $E$ does not contain an \IP\ set.	
\end{alphatheorem} 

A slight generalisation of this result is a key component in the companion paper \cite{ByszewskiKonieczny}, which proves a rigid structure theorem for all sequences which are simultaneously generalised polynomial and automatic (relevant definitions can be found therein).

Our methods rely on the link between the generalised polynomials and orbits on nilmanifolds that was established by Bergelson and Leibman (Theorem \ref{BLnilgenpolythm}). Specifically, we show that the set of times at which a linear orbit on a nilmanifold visits a semialgebraic set  cannot be an \IPS\ set, unless the set has nonempty interior (Theorem \ref{thm:S:main-1}). (The notion of an \IPS\ set generalises the notion of a shifted \IP\ set, see Definition \ref{def:IPS-set}.)

We denote the fractional part of a real number $r$ by $\fp{r}=r-\lfloor r\rfloor $, and the nearest integer to $r$ by $\lbr r \rbr = \lfloor r+1/2\rfloor$. 

We are interested in ``natural'' examples of sparse generalised polynomials. One of the simplest examples is the set of integers $n$ such that $\fp{n \varphi} < 1/n$, where $\varphi = (1 + \sqrt{5})/2$ is the golden ratio. This set is closely connected to the Fibonacci numbers. It is not difficult to show that sets of the form $\set{n \in \N}{ 0 < \fp{q(n)} < n^{-c} }$ are sparse generalised polynomials if $q(n)$ is a generalised polynomial and $c$ is rational and positive. For sequences of this form, a much simpler proof of Theorem \ref{thm:main-A} is possible (Proposition \ref{prop:IP-auto-vs-special-gp}).

Because of the intimate connection between continued fractions and diophantine approximation, it is not particularly surprising that for a quadratic irrational $\beta > 1$ of norm $\pm 1$, the set $\set{ \br{\beta^i} }{i \in \NN_0}$ is generalised polynomial (see Example \ref{ex:lin-rec-is-gp} and Proposition \ref{thm:quadraticPisotisGP}). It is considerably more surprising that a similar result holds also for certain algebraic integers of degree three (see Theorem \ref{LAcrime}).

\begin{alphatheorem}\label{thm:main-B} Assume that $\beta>1$ is either: \begin{enumerate}
\item a root of an equation $\beta^2=a\beta \pm 1$, where $a$ is an integer; or
\item  the unique real root of an equation $\b^3 = a \b^2 + b \b + 1$, where $a,b$ are integers and either \textup{(}$a \geq 0$ and $0 \leq b \leq a+1$\textup{)} or \textup{(}$a \geq 2$ and $b = -1$\textup{)}.\end{enumerate}
 Then the set $\set{ \br{\b^i}}{i \in \NN_0}$ is generalised polynomial.
\end{alphatheorem}

In the companion paper \cite{ByszewskiKonieczny}, we show that if there are non-ultimately periodic $k$-automatic sequences which are simultaneously generalised polynomials, then the set of powers of $k$ is generalised polynomial. In light of these examples, we find it pertinent to ask for which $\b > 1$ is the set $\set{ \br{\b^i}}{i \in \NN_0}$ generalised polynomial (Question \ref{Q:Pisot}). The question is of particular interest if $\b$ is an algebraic integer, or, more specifically, a Pisot number.

The final main result of this paper says that any sufficiently sparse set is generalised polynomial. For this reason, it seems unlikely that a comprehensive understanding of sparse generalised polynomials is possible. This result is perhaps expected. In \cite[Example 3.4.2]{BergelsonLeibman2007}, the set $S_a=\{n\in \N \mid 0 < \br{an} <1/n\}$ is considered for $a\in \R$. The authors note that for the choice of $a=\sum_{n=1}^{\infty} 2^{-(2^n-1)}$, we have $2^{2^n-1}\in S_a$ for $n\in \N$. However, in this construction it is not clear if $S_a$  contains any other elements. We show that the construction can be modified in order to produce the following general result.

\begin{alphatheorem}\label{thm:main-C}
	There exists a constant $C > 0$ such that for any sequence $(n_i)_{i \geq 0}$ such that $n_0 \geq 2$ and $n_{i+1} \geq n_{i}^C$ for all $i \geq 0$, the set $E = \set{n_i}{i \in \NN_0}$ is generalised polynomial.
\end{alphatheorem}

The paper is organised as follows. In Section \ref{sec:DEF} we review background material; in particular we discuss connections between dynamics on nilmanifolds and generalised polynomials. In Section \ref{sec:IPS-GP} we introduce \IPS\ sets and formulate a stronger version of Theorem \ref{thm:main-A}. In Section \ref{sec:S-GENPOLY} we prove Theorem \ref{thm:main-A}, and obtain some results on the distribution of fractional parts of polynomials on nilmanifolds which are possibly of independent interest. In Section \ref{sec:Examples} we discuss a specific class of sparse generalised polynomials related to small fractional parts, provide examples of such sequences, and prove Theorem \ref{thm:main-C}. In Section \ref{sec:Exponential} we discuss sets of values of certain sequences of exponential growth, and we prove Theorem \ref{thm:main-B}.

\subsection*{Acknowledgements.} 

The authors thank Ben Green for much useful advice during the work on this project, Vitaly Bergelson and Inger \Haland\ Knutson for valuable comments on the distribution of generalised polynomials, and Jean-Paul Allouche for comments on related results on automatic sequences.

Thanks also go to Jennifer Brown, Sean Eberhard, Dominik Kwietniak, Freddie Manners, Rudi Mrazovi\'{c}, Przemek Mazur, Sofia Lindqvist, and Aled Walker for many informal discussions.

This research was supported by the National Science Centre, Poland (NCN) under grant no. DEC-2012/07/E/ST1/00185. JK also acknowledges the generous support from the Clarendon Fund and SJC Kendrew Fund for his doctoral studies.

Finally, we would like to express our gratitude to the organizers of the  conference \emph{New developments around $\times 2\ \times 3$ conjecture and other classical problems in Ergodic Theory} in Cieplice, Poland in May 2016 where we began our project.
 
\section{Background and Definitions} \label{sec:DEF}

\subsection*{Notations and generalities}

We denote the sets of positive integers and of nonnegative integers by $\N=\{1,2,\ldots\}$ and $\N_0=\{0,1,\ldots\}$. We denote by $[N]$ the set $[N]=\{0,1,\ldots, N-1\}.$ We use the Iverson's convention: whenever $p$ is any sentence, we denote by $\ifbra{p}$ its logical value ($1$ if $p$ is true and $0$ otherwise). We denote the number of elements of a finite set $A$ by $|A|$.

For a real number $r$, we denote by $\lfloor r \rfloor$ its integer part,  by $\{r\}=r-\lfloor r\rfloor$ its fractional part, by $\lbr r \rbr = \lfloor r+1/2\rfloor$ the nearest integer to $r$, and by $\fpa{r} = |r-\lbr r\rbr|$ the distance from $r$ to the nearest integer.

We use some standard asymptotic notation. Let $f$ and $g$ be two functions defined for sufficiently large integers. We say that $f=O(g)$ or $f \ll g$ if there exists $c>0$ such that $|f(n)|\leq c \abs{g(n)}$ for sufficiently large $n$. We say that $f=o(g)$ if for every $c >0$  we have  $|f(n)|\leq c |g(n)|$ for sufficiently large $n$. 

For a subset $E\subset \N_0$, we say that $E$ has natural \emph{density} $d(A)$ if $$\lim_{N\to \infty} \frac{|E\cap [N]|}{N} = d(A).$$

We now formally define generalised polynomials. 

\begin{definition}[Generalised polynomial]
	The family of \emph{generalised polynomials} is the smallest set of functions $f \colon \ZZ \to \RR$ containing the polynomial maps and closed under addition, multiplication, and the operation of taking the integer part. A set $E\subset \Z$ is called \emph{generalised polynomial} if its characteristic function given by $f(n)=\ifbra{n\in E}$ is a generalised polynomial. 
	
	It is sometimes convenient to consider instead generalised polynomials as functions defined on the set $\N_0$ (and similarly consider generalised polynomial sets $E\subset \N_0$). The precise usage will always be clear from the context.
	\end{definition}

The class of generalised polynomial sets is clearly closed under Boolean operations.

\subsection*{Dynamical systems}
An (invertible, topological) dynamical system is given by a compact metrisable space $X$ and a continuous homeomorphism $T\colon X\to X$. We say that $X$ is transitive if there exists $x\in X$ whose orbit $\{T^n x \mid n\in \Z\}$ is dense in $X$. We say that $X$ is minimal if for every point $x\in X$ the orbit $\{T^n x \mid n\in \Z\}$ is dense in $X$. (Equivalently, the only closed subsets $Y\subset X$ such that $T(Y)\subset Y$ are $Y=X$ and $Y=\emptyset$.) We say that $X$ is totally minimal if the system $(X,T^n)$ is minimal for all $n\geq 1$. 

Let $(X,T)$ be a dynamical system. We say that a Borel measure $\mu$ on $X$ is invariant if for every Borel subset $A\subset X$ we have $\mu(T^{-1}(A))=\mu(A)$. By the Krylov-Bogoliubov Theorem, each dynamical system has at least one invariant measure. We say that a dynamical system in \emph{uniquely ergodic} if it has exactly one invariant measure. We say that an invariant measure $\mu$ on $X$ is ergodic if for every $A\subset X$, $T^{-1}(A)=A$ implies $\mu(A)\in\{0,1\}$. 

If $(X,T)$ is minimal, $x \in X$, and $U \subset X$ is open, then the set $\set{n \in \ZZ}{T^n x \in U}$ is syndetic, i.e.\ has bounded gaps (cf. \cite[Thm.\ 1.15]{Furstenberg1981}).

We will frequently use the fact that a minimal system $(X,T)$ with $X$ connected is totally minimal. (Otherwise, there exists $d \in \NN$ and a closed subset $\emptyset \neq Y \subsetneq X$ such that $T^d(Y) \subset Y$. By the Kuratowski-Zorn Lemma there exists a minimal $Y$ with this property. The sets $Y,T(Y),T^{2}(Y),\dots,T^{d-1}(Y)$ are disjoint, closed, and their union is  $X$. This contradicts $X$ being connected.)

\begin{theorem}[Ergodic theorem for uniquely ergodic systems, \cite{EinsiedlerWard}, Theorem 4.10]\label{thm:ergo-thm-uniform}
	Let $(X,T)$ be a uniquely ergodic dynamical system with an invariant measure $\mu$. Then, for any continuous function $f\colon X\to \CC$, we have 
	$$
		\frac{1}{N} \sum_{n = 0}^{N - 1} f(T^n x) \to \int_X f d\mu
	$$
	as $N \to \infty$. Furthermore, the convergence is uniform in $x \in X$.
\end{theorem} 

By a standard argument that bounds the characteristic function of a set from above and from below by continuous functions, we obtain the following result, frequently used in the sequel.
\begin{corollary}\label{cor:density-uniform}
	Let $(X,T)$ be a uniquely ergodic dynamical system with an invariant measure $\mu$. Then for any $x \in X$ and any $S \subset X$ with $\mu(\partial S) = 0$, the set $E = \set{n \in \NN_0}{T^n x \in S}$  has density $\mu(S)$.
\end{corollary}

\subsection*{Generalised polynomials and nilmanifolds}\label{sec:DEF-GP}

A nilmanifold is a homogenous space $X=G/\mathord \Gamma$, where $G$ is a nilpotent Lie group and $\Gamma$ is a discrete cocompact subgroup, together with the action  of $G$ on $X$ via left translations. We do not assume that $G$ is connected, however in all the cases that we will consider the connected component $G^{\circ}$ of $G$ will be \emph{simply connected}. We assume this henceforth, since it will substantially simplify the discussion.

We begin by recalling a few basic facts concerning nilpotent Lie groups. We follow the presentation in \cite{BergelsonLeibman2007} and \cite{Malcev1951} which the reader should consult for the proofs. Let $G$ be a connected simply connected nilpotent Lie group. For each $g\in G$, there is a unique one-parameter subgroup $\{g^t\}_{t\in\R}$ of $G$, i.e. a continuous homomorphism $\R \to G$, $t\mapsto g^t$ with $g^1=g$. Consider the lower central series 
$$ G_0 = G_1 \supset G_2 \supset \ldots\supset G_d\supset G_{d+1}=1$$ 
given by $G = G_0 = G_1$ and $G_{i+1}=[G,G_i]$,
 $1\leq i \leq d$. The subgroups $G_i$ are closed and $G_i/\mathord G_{i+1}$ are finite dimensional $\R$-vector spaces with the action of $\R$ given by one-parameter subgroups. 
 
 Let $l_i=\dim_{\R} G_i/\mathord G_{i+1}$, $k_i=\sum_{j=1}^i l_j$, $0\leq i \leq d$. Then $G$ has a \emph{\Malcev\ basis}, i.e. elements $e_1,\ldots,e_k\in G$, $k=k_d$, such that $e_{k_{i-1}+1},\ldots,e_{k_{i}}\in G_i$ and their images in $G_i/\mathord G_{i+1}$ constitute a basis of $G_i/\mathord G_{i+1}$ as a $\R$-vector space.

It follows easily that any element $g\in G$ can be written uniquely in the form $$g=e_1^{t_1}\cdots e_k^{t_k}, \quad t_i\in \R.$$ Furthermore, it is possible to choose a \Malcev\ basis to be compatible with $\Gamma$, i.e. so that $g=e_1^{t_1}\cdots e_k^{t_k}$ is in $\Gamma$ if and only if $t_1,\ldots,t_k \in\Z$. We will always assume that our \Malcev\ bases are compatible with $\Gamma$. A choice of a \Malcev\ basis determines a diffeomorphism $$ \tilde \tau \colon  G \to \R^k, \quad  e_1^{t_1}\cdots e_k^{t_k} \mapsto (t_1,\ldots, t_k).$$
Under this identification, multiplication in $G$ is given by polynomial formulae in variables $t_i$ with rational coefficients. These polynomials take integer values if all the variables are integers.

Let $[0,1)^k=\{(t_1, \ldots,t_k)\in \R^k \mid 0\leq t_i < 1, 1\leq i\leq k\}$ be the unit cube. The preimage $D = \tilde{\tau}^{-1}([0,1)^k) \subset G$ is the \emph{fundamental domain}. For any $g\in G$, there is a unique choice of elements $ \{g \}\in D, [g]\in \Gamma$ (called the \emph{fractional part}  and the \emph{integral part} of $g$) such that $g=\fp{g} \ip{g}$. Note that the elements $\{g\}$ and $[g]$ depend on the choice of the \Malcev\ basis. Since the \Malcev\ basis is compatible with $\Gamma$, the map $g\mapsto \tilde{\tau}(\fp{g})$ factors to a bijection
$$ \tau \colon  G/\Gamma \to [0,1)^k,$$
which will play a crucial r\^ole. While $\tau$ is not continuous, its inverse $\tau^{-1}$ is.

The choice of the \Malcev\  basis also induces a natural choice of a metric on $G$ and $G/\Gamma$ (cf.\ \cite[Definition 2.2]{GreenTao2012}). 

A \emph{horizontal character} on $X=G/\mathord \Gamma$ is a non-trivial morphism of groups $\eta \colon G \to \RR/\ZZ$ with $\Gamma \subset \eta^{-1}(0)$, and is necessarily of the form $\eta(g) = \sum_{i=1}^{l_1} a_i \tau_i(g\Gamma)$, where $\tau=(\tau_1,\ldots,\tau_k)$ and  $a_i \in \ZZ$ are not all $0$. The \emph{norm} of $\eta$, denoted $\norm{\eta}$, is by definition the Euclidean norm $\norm{ (a_i) }_2$. 

A set $A \subset \R^k$ is called semialgebraic if it is a finite union of subsets of $\R^k$ given by a system of finitely many polynomial equalities and inequalities; in particular, $A$ is Borel and either $A$ has nonempty interior or it is of Lebesgue measure zero. A map $p\colon A \to \R^k$ defined on a semialgebraic set $A\subset \R^k$ is called  \emph{piecewise polynomial} if there is a decomposition $A=A_1\cup \ldots \cup A_s$ of $A$ into a finite union of semialgebraic sets such that $p|_{A_i} \colon A_i \to \R^l$ is a polynomial mapping restricted to $A_i$. We call a map $f\colon G \to \R^l$ a \emph{polynomial map} if the composed map $f \circ \tilde{\tau}^{-1} \colon\R^l \to \R^k$ is a polynomial map. While the map $\tilde\tau\colon G \to \R^k$ depends on the choice of a \Malcev\ basis, the concept of a polynomial map on $G$ does not.

A subset $A\subset X= G/\mathord \Gamma$ is called \emph{semialgebraic} if its image $\tau(A)\subset [0,1)^k$ is semialgebraic.  We call a map $p\colon X \to \R^l$ \emph{piecewise polynomial} if it takes the form $p=q \circ \tau$, where $q\colon [0,1)^k \to \R^l$ is a piecewise polynomial map. As before, the concept of a piecewise polynomial map on $X$ does not depend on the choice of a \Malcev\ basis. Note that if $p\colon X \to \R^l$ is a piecewise polynomial map and $g\in G$, then so is $q\colon X \to \R^l$ given by $q(h\Gamma)=p(gh\Gamma)$. 

We now extend the above definitions to not necessarily connected simply connected Lie groups $G$. A \Malcev\ basis of $G$ is simply a \Malcev\ basis of its connected component $G^{\circ}$. Similarly, we define piecewise polynomial maps $p\colon X \to \R^l$ in terms of their restrictions to the connected components of $X$ by demanding that for each $g\in G$ the map $G^{\circ}/\mathord (\Gamma \cap G^{\circ}) \to \R^l$, $h(\Gamma \cap G^{\circ})\mapsto p(gh\Gamma)$, is piecewise polynomial. (It is enough to verify this condition for a single $g$ in each class $g G^{\circ} \in G/\mathord G^{\circ}$.) We extend the notion of the fractional and the integral part to the case of possibly disconnected $G$, but we define and use it only if the nilmanifold $G/\mathord \Gamma$ is nevertheless connected. In this case the map $G^{\circ} \to G/\mathord \Gamma$ is surjective and induces a diffeomorphism $G^{\circ}/\mathord \Gamma' \to G/\mathord \Gamma,$ where $\Gamma'=\Gamma\cap G^{\circ}$. Thus, we may define the fractional part with respect to $G^{\circ}/\mathord \Gamma'$, and define the integral part of $g\in G$ so that $g=\fp{g} \ip{g}$. (We can do this since these notions do not involve dynamics, and depend only on the nilmanifold; in dynamical considerations, we are still interested in the action of the (disconnected) group $G$ on the nilmanifold $X=G/\mathord \Gamma \simeq G^{\circ}/\mathord \Gamma'$.)

Dynamics on nilmanifolds plays an important r\^ole. Any $g \in G$ acts on $X$ by left translation $T_g(h\Gamma) = gh\Gamma$. There is a unique Haar measure $\mu_{X}$ on $G/\mathord \Gamma$ invariant under the left translations. For any $g\in G$, we obtain a dynamical system $(X,T_g)$. By a result of Parry \cite{Parry1969}, the properties of minimality, unique ergodicity, and ergodicity with respect to the measure $\mu_X$ are equivalent. (In fact, all this conditions can be verified on the maximal torus $G/G_2\Gamma \simeq \TT^{l_1}$.)  Furthermore, the action of $G$ on $X$ is distal, i.e. for $x,y\in X $, $x\neq y$, $\inf_{g\in G} \mathrm{dist}(gx,gy)$ is strictly positive.

Generalised polynomials and orbits on nilmanifolds are intimately related. This is explained in particular by the main result of \cite[Theorem A]{BergelsonLeibman2007}.

\begin{theorem}[Bergelson-Leibman]\label{BLnilgenpolythm}\label{thm:BergelsonLeibman}
 \begin{enumerate} 
\item If $X=G/\mathord \Gamma$ is a nilmanifold, $g\in G$ acts on $X$ by left translations, $p\colon X \to \R$ is a piecewise polynomial map, and $x\in X$, then $u \colon \Z \to \R$ given by $u(n)=p(g^n x), n\in \Z$, is a bounded generalised polynomial.
\item If $u\colon \Z \to \R$ is a bounded generalised polynomial, then there exists a nilmanifold $X=G/\mathord \Gamma$, $g\in G$ acting on $X$ by left translations in such a way that the action is ergodic, a piecewise polynomial map $p\colon X \to \R$, and $x\in X$ such that $u(n)=p(g^n x)$, $n\in \Z$.\end{enumerate}\end{theorem}

We finish by defining \emph{polynomial sequences} with values in nilpotent groups. Strictly speaking, polynomial sequences are defined relative to a filtration; we only work with polynomial sequences with respect to the lower central series.

\begin{definition} Let $G$ be a nilpotent group and let $G= G_0 = G_1 \supset G_2 \supset \ldots \supset G_d \supset G_{d+1}=1$ be its lower central series. A sequence $g\colon \ZZ \to G$ is polynomial if it takes the form
	$$
		g(n) = g_0^{\binom{n}{0}} g_1^{\binom{n}{1}} g_2^{\binom{n}{2}} \dots g_d^{\binom{n}{d}}$$ for some $g_i \in G_i,\ 0 \leq i \leq d$.
\end{definition}

It has been proven by Lazard and Leibman that polynomial sequences form a group with termwise multiplication. For more details, see \cite{Lazard-1954}, \cite{Leibman1998}, \cite{Leibman2002}. This is one of the many reasons why it is often more convenient to work with polynomial sequences, even if one is ultimately interested in linear orbits.

Distribution properties of polynomial sequences have been extensively studied. In the quantitative setting we have the following result. We will also later need a quantitative variant due to Green and Tao (cf.\ Theorem \ref{thm:GreenTao}).

\begin{theorem}[Leibman, \cite{Leibman-2005b}]
	Let $G/\Gamma$ be a nilmanifold, and let $g \colon \ZZ \to G$ be a polynomial sequence. Then either $(g(n)\Gamma)_{n \geq 0}$ is equidistributed in $X$ (with respect to $\mu_X$), or there exists a horizontal character $\eta$ such that $\eta \circ g$ is constant.
\end{theorem}

\subsection*{\IP\ sets and \IP\ convergence}

We denote by $\cF$ the set of all finite nonempty subsets $\a \subset \NN$.  For sets $\alpha,\beta \in \cF$, we write $\alpha < \beta$ if $a<b$ for all $a\in \alpha, b\in \beta$. We will extensively use sequences $(n_{\alpha})_{\alpha \in \cF}$ indexed by finite sets of natural numbers. In fact, for a sequence $(n_i)_{i\in \N}$ of integers, we will frequently use the associated sequence $(n_{\alpha})_{\alpha \in \cF}$ given by $n_{\alpha}=\sum_{j\in \alpha} n_j$ and denote its set of finite sums by $\FS(n_i) = \set{ n_\a }{ \a \in \cF}.$

The following notion is widely used and well-studied.

\begin{definition}[\IP\ set]\label{def:IP-set} A set $E\subset \N$ is called an \emph{\IP\ set} if it contains a set of the form  $\FS(n_i)$ for some sequence of natural numbers $(n_i)_{i\in \N}$. Similarly, $E \subset \NN$ is an \IPP\ set if it contains $E_0 + a$ for some $a \in \N_0$ and some \IP\ set $E_0$.  

A set $E\subset \N$ is called an \emph{\IPd\ set} if it intersects nontrivially any \IP\ set.
\end{definition}

The class of \IP\ sets is partition regular, i.e.\ whenever an \IP\ set $E$ is written as a union of finitely many subsets $E=E_1\cup E_2 \cup \ldots \cup E_r$, at least one of the sets $E_i$ is an \IP\ set. This is the statement of Hindman's Theorem (see \cite{Hindman1974}, \cite{Bergelson2010d}, or \cite{Baumgartner-1974}). It follows that an \IPd\ set is an \IP\ set, an intersection of two \IPd\ sets is an \IPd\ set, and an intersection of an \IPd\ set and an \IP\ set is an \IP\ set. (For the first and the third statement, let $A$ be an \IPd\ set and let $B$ be an \IP\ set. Apply Hindman's theorem to the partitions $\N=A \cup (\N\setminus A)$ and  $B=(B\cap A) \cup (B\setminus A)$. The second statement then easily follows.)

\IP\ sets occur naturally in dynamics. The following statement is sometimes known as the \IP\ Recurrence Theorem. Recall that a topological dynamical system $(X,T)$ is distal if  for $x,y\in X$, $x \neq y$, $\inf_{n\in \Z} \mathrm{dist}(T^n x,T^n y) $ is strictly positive. The action by left translation of $g\in G$ on a nilmanifold $X=G/\mathord \Gamma$ is distal.

\begin{theorem}[\cite{Furstenberg1981}, Lemma 9.10]\label{thm:IPrecurrence}
 Let $(X,T)$ be a minimal distal topological dynamical system. Then for every $x\in X$ and every neighbourhood $U$ of $x$, the set of return times $\{n\in \N \mid T^n(x)\in U\}$ is an \IPd\ set.
\end{theorem}

A related concept is that of an \IP\ ring. A family $\mathcal{G}\subset \cF$ of subsets of $\N$ is called an \emph{\IP\ ring} if there exists a sequence $\beta=(\beta_i)_{i\in \N}\subset \cF$ with $\beta_1<\beta_2<\ldots$ and such that $$\mathcal{G}=\set{\bigcup_{i \in \gamma} \beta_i }{ \gamma \in \cF}.$$  An equivalent version of Hindman's theorem says that whenever  an \IP\ ring $\mathcal{G}$ is written as a union of finitely many subfamilies $\mathcal{G}=\mathcal{G}_1\cup \mathcal{G}_2 \cup \ldots \cup \mathcal{G}_r$, at least one of the subfamilies $\mathcal{G}_i$ contains an \IP\ ring (cf.\ \cite{Bergelson2010d}, \cite{Baumgartner-1974}.) 

\newcommand{\iplim}{\mathrm{IP-}\!\lim}

There is a natural notion of convergence for sequences indexed by $\cF$, or more generally an \IP\ ring. Let $X$ be a topological space and let $(x_\alpha)_{\alpha \in \cF}$ be a sequence of points of $X$. We say that the sequence $(x_\alpha)_{\alpha \in \cF}$ converges to a limit $x\in X$ along an \IP\ ring $\mathcal{G}$ and we write $$\iplim_{\a \in \mathcal{G}} x_{\alpha}  = x$$ if for any neighbourhood $U$ of $x$ there exists $\alpha_0\in \cF$ such that for all $\alpha \in \mathcal{G}$ with $\alpha>\alpha_0$ we have $x_{\alpha}\in U$. Limits along \IP\ rings are closely related to limits along idempotent ultrafilters in $\N$.

Let $X$ be a compact topological space and let $(x_\alpha)_{\alpha \in \mathcal{F}}$ be a sequence of points of $X$. It is a corollary of Hindman's theorem and a diagonal argument (\cite[Theorem 8.14]{Furstenberg1981} or \cite[Theorem 1.3]{FurstenbergKatznelson-1985}) that there exists an \IP\ ring $\mathcal{G}$ and a point $x\in X$ such that $\iplim_{\a \in \mathcal{G}} x_{\alpha}  = x$.

\section{\IPS\ sets and sparse generalised polynomials}\label{sec:IPS-GP}

In this section, we discuss a slight generalisation of Theorem \ref{thm:main-A} which we will need in the upcoming paper \cite{ByszewskiKonieczny}. We also discuss how it can  be derived from the results concerning nilsystems proved in Section \ref{sec:S-GENPOLY}.

\begin{definition}[\IPS\ set]\label{def:IPS-set}
We say that $E \subset \NN$ is an \IPS\ set if there exist sequences $(n_i)_{i \in \NN} \subset \NN$, $(N_t)_{t \geq 1} \subset \NN_0$ and $(t_0(\a))_{\a \in \cF} \subset \NN$ such that $n_\a + N_t \in E$ for any $\a \in \cF$ and $t \geq t_0(\a)$.
\end{definition}

A translate of an \IP\ set is an \IPS\ set, but not conversely. The interest in \IPS\ sets stems mainly from the fact that automatic sets which do not contain \IPS\ sets can be completely classified (see \cite{ByszewskiKonieczny}; a similar statement for \IPP\ sets does not hold). 

\begin{remark} \IPS\ sets are very natural to consider from the point of view of ultrafilters. The \v{C}ech-Stone compactification $\beta \N$ of $\N$ (regarded as a discrete topological space) consists of ultrafilters on $\N$. There is a natural associative operation on $\beta \N$ which gives it a structure of a left semitopological semigroup. It is not commutative, and the center consists of principal ultrafilters (identified with elements of $\N$). 

It is well known that a set $E\subset \N$ is an \IP\ set if and only if it is belongs to an ultrafilter $p$ that is idempotent (i.e. $p+p=p$). One can immediately conclude that a set is an \IPP\ set if and only if  it belongs to an ultrafilter of the form $n+p=p+n$ for some $n\in \N$ and some idempotent $p\in \beta \N$.  One can show that a set is an  \IPS\ set if and only if it belongs to an ultrafilter of the form $r+p$ for some $r\in \beta \N$. The proofs are quite immediate. (For the latter statement, note that $E \subset \N$ is an \IPS\ set if and only if there is a sequence $(n_i)_{i \in \NN} \subset \NN$ such that the family $\{ (E-n_{\alpha}) \cap \NN \}_{\alpha \in \cF}$ has the finite intersection property. We choose $p, r \in \beta \N$ such that $p+p=p$, $\FS(n_i) \in p$, and $(E-n_{\alpha}) \cap \NN \in r$.)

Thus a set is an \IPP\ set if and only if it belongs to  an ultrafilter lying in the right ideal generated by an idempotent ultrafilter and a set is an \IPS\ set if and only if it belongs to  an ultrafilter lying in the left ideal (or equivalently the two-sided ideal) generated by an idempotent ultrafilter.  For more details about $\beta \N$, see \cite{Bergelson2010d} and references therein. We will not use the ultrafilter interpretation of \IPS\ sets in what follows.
\end{remark}

The following theorem is a strengthening of Theorem \ref{thm:main-A}.

\begingroup
\def\thealphatheorem{A'}
\begin{alphatheorem}\label{thm:main-A1}
	Suppose that $E$ is a sparse generalised polynomial set. Then $E$ does not contain an \IPS\ set.	
\end{alphatheorem} 
\endgroup

Any \IPS\ set contains a translate of an $\mathrm{IP}_r$\ set (i.e.\ the set of all finite sums of a finite sequence $n_1,\dots,n_r$) for any $r$. It  is therefore natural to ask if Theorem \ref{thm:main-A1} can be generalised to $\mathrm{IP}_r$\ sets. The answer is negative, as shown by the following example.

\begin{example}
	Let $E = \set{ n \in \N}{ \fpa{n\sqrt{2}} < 1/\sqrt{n}}$. Then $E$ is a sparse generalised polynomial set and $E$ contains sets of the form $\{ km \mid 1\leq k \leq r\}$ for arbitrary $r$. In particular, $E$ contains an $\mathrm{IP}_r$ set.
\end{example}
\begin{proof}
	By standard diophantine approximation, there exist arbitrarily large $m$ such that $\fpa{m \sqrt{2} } < 1/m$. Then $km \in E$ for $k \leq m^{1/3}$. Sparseness of $E$ follows from Proposition \ref{prop:special-gp-is-gp}.
\end{proof}

\begin{remark}\label{BL-implies-weak-A}
	A weaker version of Theorem \ref{thm:main-A} follows immediately by an adaptation of the methods in \cite{BergelsonLeibman2007}. Indeed, adapting the proof of Theorem C therein, one concludes that if $E$ is a sparse generalised polynomials set, then for all $n$ except for a set of (upper Banach) density $0$ the set $\NN \setminus (E-n)$ is {\IPd}, whence $E - n$ contains no \IP\ set. Unfortunately, this weaker statement is not as useful in applications; in particular is not sufficient for the purposes of \cite{ByszewskiKonieczny}.
\end{remark}
	 
\section{The main argument}\label{sec:S-GENPOLY}

\newcommand{\PHI}{P}\newcommand{\PSI}{Q}
In this section, we prove Theorem \ref{thm:main-A}. In the process, we obtain results concerning the distribution of fractional parts of polynomial sequences, such as Propositions \ref{lem:variable-separation} and \ref{lem:fp-closure}, which are possibly of independent interest.

It will be convenient to reformulate Theorem \ref{thm:main-A}. By Theorem \ref{thm:BergelsonLeibman}, a sparse generalised polynomial set $E$ arises as $u(n)=p(g^n a)$ for a nilmanifold $X=G/\mathord \Gamma$, $g\in G$ acting on $X$ ergodically, a piecewise polynomial map $p\colon X \to \R$, and $a\in X$. Let $S=\{x\in X \mid p(x)= 1\}$. Since $E$ is sparse, Corollary \ref{cor:density-uniform} shows that $\inter S =\emptyset$. (As a side remark, we note that it is also easy to see that a sparse genereralised polynomial set has upper Banach density zero.) Thus Theorem \ref{thm:main-A} is equivalent to the following statement.

\begin{theorem}\label{thm:S:main-1}
	Let $X=G/\Gamma$ be a nilmanifold, let $g \in G$ be such that the left translation by $g$ is ergodic, $a \in G$, and let $S \subset X$ be a semialgebraic subset with $\inter S = \emptyset$.  
	Then there does not exist an \IPS\ set $E$ such that 
	\begin{equation}
	\label{eq:S:main-1}
	g^{n } a \Gamma \in S, \qquad n \in E.
	\end{equation}
\end{theorem}

We proceed to the proof of Theorem \ref{thm:S:main-1}.

\subsection*{Preliminaries}
In dealing with \IPS\ sets, the following lemma will be useful.

\begin{lemma}[Partition regularity for \IPS\ sets]\label{lem:IPS-partition-regular}
Suppose that $E \subset \NN$ is an \IPS\ set. Then for any finite partition $E = E_1 \cup E_2 \cup \ldots \cup E_r$ one of the sets $E_j$ is an \IPS\ set.
\end{lemma}
\begin{proof} This is clear from the ultrafilter interpretation of \IPS\ sets. We give a direct proof below.
	Suppose that $(n_i)_{i \in \NN}$, $(N_t)_{t \geq 1}$ and $(t_0(\a))_{\a \in \cF}$ are such that $n_\a + N_t \in E$ for $\a \in \cF$ and $t \geq t_0(\a)$. Restricting $N_t$ to subsequence by means of a familiar diagonal argument, we may ensure that for any $\a$ there is some $j = j(\a)$ such that $n_\a + N_t \in E_j$ for all $t \geq t_0(\a)$. The claim now follows directly from Hindman's theorem.\end{proof}

\newcommand{\fpp}[2]{\fp{#1}_{#2}}
\newcommand{\ipp}[2]{\left[ #1 \right]_{#2}}

The following lemmas are classical. 

\begin{lemma}\label{lem:G2-not-complemented}
	Let $G$ be a nilpotent group (not necessarily Lie), and let $H < G$ be a subgroup of $G$ such that $H G_2 = G$. Then $H = G$.
\end{lemma}

For the proof, see e.g. \cite[Lemma 2.5]{Leibman-2005b}, \cite[Lemma 3.4]{Leibman2005}, or \cite[Lemma 1.11]{Zorin-Kranich2013}.

\begin{lemma}\label{lem:poly-seq-in-G-o}
	Let $G/\mathord\Gamma$ be a connected nilmanifold, and assume further that all the nilmanifolds $G_i/\mathord\Gamma_i$ are connected, where $\Gamma_i=\Gamma\cap G_i$. Let $g\colon \ZZ \to G$ be a polynomial sequence. Then there exists a polynomial sequence $h \colon \ZZ \to G$ taking values in $G^{\circ}$ such that $g(n)\Gamma = h(n)\Gamma$.
\end{lemma}

This follows from \cite[Lemma 4.24]{Zorin-Kranich2013}.

\begin{remark} We do not claim that $h$ is a polynomial sequence in $G^{\circ}$, since possibly $(G^{\circ})_i \subsetneq (G_i)^{\circ}$. In general, one often defines polynomial sequences with respect to a filtration of subgroups. We will not need that, and our polynomial sequences will always be taken with respect to the lower central series filtration. This explains a bit akward formulation of the preceding lemma.\end{remark}

We will also need the following basic fact about invariant algebraic sets. 

\begin{lemma}\label{lem:invariant=>strongly-invariant}
	Suppose that $A \subset \RR^d$ is an algebraic set and $\PHI \colon \RR^d \to \RR^d$ is a surjective polynomial map. If $\PHI^{-1}(A) \subset A$, then also $\PHI(A) = A$.
\end{lemma}
\begin{proof}
Consider the descending sequence of nonempty algebraic sets
	\begin{align}
	\label{eq:865-1}
		A \supset \PHI^{-1}(A) \supset \PHI^{-2}(A) \supset \PHI^{-3}(A) \supset \dots
	\end{align}	 
	
 	By Hilbert's Basis Theorem, there is $n$ such that $\PHI^{-n}(A) = \PHI^{-n-1}(A)$. Since $P$ is surjective, we get $\PHI(A) = A$.
\end{proof}

\subsection*{Initial reductions}

Our proof of Theorem \ref{thm:S:main-1} is indirect: we exploit the existence of an orbit satisfying \eqref{eq:S:main-1} to construct algebraically invariant objects whose existence is inreasingly difficult to sustain, up to the point when we obtain a blatant contradiction. We begin with a relatively straightforward reduction.

\begin{step}\label{prop:S:main-2}
	Assume Theorem \ref{thm:S:main-1} is false. Then there exists a nilmanifold $G/\mathord\Gamma$ such that every $G_i/\mathord\Gamma_i$ is connected, where $\Gamma_i=\Gamma\cap G_i$, an element $g \in G$ acting ergodically by left translations, a semialgebraic subset $S \subset G/\Gamma$ with empty interior, and an \IP\ set $E \subset \NN$ such that
	\begin{equation}
	\label{eq:S:main-2}
	g^{n } \Gamma \in S, \qquad n \in E.
	\end{equation}
\end{step}
\begin{proof}	Suppose that Theorem \ref{thm:S:main-1} fails, so that \eqref{eq:S:main-1} holds. 
	 By \cite[Lemma 4.5]{Zorin-Kranich2013}, there exists a discrete cocompact subgroup $\Gamma < \tilde \Gamma < G$ such $[\tilde \Gamma : \Gamma] < \infty$, 
	 and such that  $G_i/\mathord{\tilde{\Gamma}_i}$ is connected for every $i$. We replace $\Gamma$ by $\tilde \Gamma$ and $S$ by its image under the map $G/\mathord\Gamma \to G/\mathord{\tilde{\Gamma}}$ (this is still a semialgebraic set with empty interior).

	 Let $(n_i)_{i \in \NN}$, $(N_t)_{t \geq 1}$ and $(t_0(\a))_{\a \in \cF}$ be such that $n_\a + N_t \in E$ for $t \geq t_0(\a)$, so that
	$$
		g^{n_{\a} + N_t} a \Gamma \in S. 
	$$
	Replacing $S$ by its closure, which is again a semialgebraic set with empty interior, we may assume that $S$ is closed. Restricting $N_t$ to a subsequence if necessary, we may assume that $g^{N_t} a \Gamma \to b \Gamma$ as $t \to \infty$ for some $b \in G$. 	It follows that for any \IP\ set $E' \subset \FS(n_i)$ we have 
	$$
		g^{n } b \Gamma = \lim_{t \to \infty} g^{n + N_t} a \Gamma  \in S, \qquad n \in E'. 
	$$
	Finally, note that the map $g\mapsto gb^{-1}$ induces an isomorphism $G/\mathord \Gamma \to G/\mathord{{\Gamma}'}$, where $\Gamma' = b \Gamma b^{-1}$. Let $S' \subset G/\mathord{{\Gamma}'}$ be the image of $S$ under this map. We find that $g^n \Gamma' \in S'$ for $n \in E'$, as needed.
\end{proof}
	  
\subsection*{Fractional parts and limits}
	  
Let $G/\Gamma$ be a nilmanifold with a specified \Malcev\ basis. A technical difficulty in dealing with fractional parts in $G$ stems from the fact that, for a sequence $g_n \in G$, the limit of $g_n \Gamma$ does not determine the limit of $\fp{g_n}$ if the limit lies on the boundary of the fundamental domain. The following lemma partially overcomes this difficulty. 
		  
\begin{lemma}\label{lem:conv-of-fp-to-0}
	Let $X=G/\mathord\Gamma$ be a connected nilmanifold. Then for any sequence $(g_n)_{n \geq 0}$ in $G$ such that $\lim_{n \to \infty} g_n \Gamma = e \Gamma$, there exists a choice of a \Malcev\ basis $\cX$ of $X$ and a subsequence $(g'_{n})$ of $(g_n)$ such that $\lim_{n \to \infty} \fp{g'_n} = e$.
	
	Likewise, for any $(g_\a)_{\a \in \cF}$ in $G$ such that  $\iplim_{a \in \cF} g_\a \Gamma = e \Gamma$, there exists a choice of a \Malcev\ basis $\cX$ and an \IP\ ring $\cG$ such that $\iplim_{\a \in \cG} \fp{g_\a} = e$.
\end{lemma}
\begin{proof} Since $G/\mathord \Gamma$ is connected, we have an isomorphism $G/\mathord \Gamma \simeq G^{\circ}/\mathord{(\Gamma\cap G^{\circ})}$, and hence we may assume that $G$ is connected. Recall that a choice of a \Malcev\ basis $\cX = (e_1,\ldots,e_k)$ determines the map $\tilde \tau=(\tilde\tau_1,\ldots,\tilde\tau_k)\colon G \to \RR^k$.
	By induction, we will prove that for each $1\leq r\leq k+1$ there exists a \Malcev\ basis $\cX$ such that, after passing to a subsequence $(g'_n)$ of $(g_n)$, we have 
	$$\lim_{n \to \infty} \fp{g'_n} = h$$
	for $h\in \cl D = \tilde \tau^{-1}([0,1]^k)$ such that $\tilde\tau_i(h) = 0$ for $1\leq i < r$. For $r = 1$, the claim is trivially satisfied, and the claim for $r = k+1$ implies the claim in the statement of the lemma.

	If the claim holds for $r$, then we may write:
	$$
		g_n' = e_r^{t_r(n)} e_{r+1}^{t_{r+1}(n)} \dots e_k^{t_k(n)}  \gamma(n),
	$$
	where $t_i(n) \in [0,1)$, $\gamma(n) \in \Gamma$, and $t_i(n) \to \tilde\tau_i(h)$ as $i \to \infty$. Since $h \in \Gamma$, we have $\tilde\tau_i(h) \in \{0,1\}$. If $\tilde\tau_r(h) = 0$, we are done. Otherwise, we may write
	\begin{align*}
	g_n' 
	&= (e_r^{-1})^{1-t_r(n)}  e_{r+1}^{t_{r+1}(n)} \dots e_k^{t_k(n)}  \gamma(n) \\
	&= (e_r^{-1})^{1-t_r(n)}  e_{r+1}^{t_{r+1}'(n)} \dots e_k^{t'_k(n)} \gamma'(n),
	\end{align*}
	where $t_i'(n) \in [0,1)$ and $\gamma'(n) \in \Gamma$. Replacing $\cX$ with $\cX' =(e_1',\dots,e_k')$ with $e_r' = e_r^{-1}$ and $e_i' = e_i$ for $i \neq r$, we find the sought \Malcev\ basis for $r+1$.
	
	The proof of the claim concerning  \IP\ limits is completely analogous.
\end{proof}

We are now ready to perform the first substantial step in the proof of Theorem \ref{thm:S:main-1}. We construct an algebraic set inside $G$ with surprising invariance properties. This is helpful largely because working in $G$ is easier than in $G/\Gamma$.

\begin{step}\label{prop:S:main-3} Suppose that Theorem \ref{thm:S:main-1} is false. Then there exists a nilmanifold $G/\mathord\Gamma$ with a \Malcev\ basis $\cX$ such that every $G_i/\mathord\Gamma_i$ is connected, where $\Gamma_i=\Gamma\cap G_i$, an element $g \in G$ acting ergodically by left translations, a nonempty algebraic subset $R \subset G^{\circ}$ with empty interior, and an \IP\ set $E = \FS(n_i)$ such that $R$ is preserved under the operations
\begin{equation}
		\PHI_\a \colon G^{\circ} \to G^{\circ}, \qquad x \mapsto g^{n_\a} x \ip{ g^{n_\a}}^{-1}
\end{equation}  
	for all $\a \in \cF$. \end{step}

\begin{proof}
	Let the notation be as in  Step \ref{prop:S:main-2}. We will subsequently pass to \IP\ subsets of $E = \FS(n_i)$ to ensure better properties.
	
	After replacing  $E$ with a smaller \IP\ set, we may assume that $\iplim_{\a \in \cF} g^{n_\a} \Gamma = e \Gamma$. (Indeed, we pass to an \IP\ subring on which the sequence is convergent and apply the \IP\ Recurrence Theorem \ref{thm:IPrecurrence}). Hence, by Lemma \ref{lem:conv-of-fp-to-0} there exists a choice of a \Malcev\ basis such that, possibly after shrinking $E$ again, we have 
\begin{equation}
	\label{eq:527-1}
	\iplim_{\a \in \cF} \fp{ g^{n_\a} } = e.
\end{equation}  
	
For a fixed $\b$ and sufficiently large $\a > \b$, we aim to compute $\fp{g^{n_{\a \cup \b}}}$ in terms of $g^{n_\b}$ and $\fp{g^{n_\a}}$. We begin by noting that
\begin{equation} \label{eq:527-2}
	\fp{g^{n_{\a \cup \b}}} 
	= \fp{ g^{n_\b} \fp{g^{n_\a}}} 
	 	= g^{n_\b} \fp{g^{n_\a}} \ip{g^{n_\b}}^{-1} \gamma(\a,\b),
\end{equation}
where $\gamma(\a,\b)=\ip{g^{n_\b}} \ip{g^{n_\b}\fp{g^{n_\a}}}^{-1}\!\!\in \Gamma$ takes, for a fixed $\b$, only finitely many possible values. For a fixed $\beta$, we may use Hindman's theorem to pass to an \IP\ subring so that $\gamma(\a,\b)$ does not depend on the choice of $\a > \b$. Using a standard diagonal argument, we replace $E$ by an \IP\ subset so that $\gamma(\a,\b) = \gamma(\b)$ does not depend on $\a>\b$ for all $\b \in \cF$. Letting $\a \to \infty$ in (\ref{eq:527-2}), we get
$$
	\iplim_{\a \in \cF} \fp{g^{n_{\a \cup \b}}} 
= 	\fp{ g^{n_\b} } \gamma(\b).
$$ 
	If $\b$ is sufficiently large, this implies that $\gamma(\b) = e$. Passing to an \IP\ subset of $E$, we may assume that $\gamma(\beta) = e$ for all $\b \in \cF$. Hence, applying (\ref{eq:527-2}) again, we have
\begin{align}
	\label{eq:527-3}
	\fp{g^{n_{\a \cup \b}}} =  g^{n_\b} \fp{g^{n_\a}} \ip{g^{n_\b}}^{-1} = \PHI_\b(\fp{g^{n_\a}})
\end{align}
	for any $\a,\b \in \cF$ with $\b < \a$.	
	
	We claim that $\PHI_\b \circ \PHI_\a = \PHI_{\a \cup \b}$ for all $\a > \b$. Indeed, both maps take the form $x\mapsto g^{n_{\a\cup\b}}x \gamma$ for some $\gamma \in \Gamma$ and by (\ref{eq:527-3}) agree on points $\fp{g^{n_{\lambda}}}$ for $\lambda > \alpha$. Therefore, the two maps are equal. We will use this fact below.

	Denote by $\widetilde{S} \subset G^{\circ}$ the Zariski closure of the copy of $S$ contained in the fundamental domain for $\cX$. Then $\tilde S$ is a algebraic subset of $G^{\circ}$ with empty interior.	 For any $\b \in \cF$, let $R_\b$ denote the set of $x \in \tilde{S}$ such that $\PHI_{\b}(x) \in \tilde S$. Clearly, these sets are algebraic, and by \eqref{eq:527-3} we have $\fp{ g^{n_\a} } \in R_\b$ for $\a > \b$.

	Let $R = \bigcap_{\a \in \cF} R_\a$.	By Hilbert's Basis Theorem, there is a finite collection $\cB \subset \cF$ such that $R = \bigcap_{\b \in \cB} R_\b$. Note that $R$ is nonempty as $\fp{ g^{n_\a} } \in R$ for $\a$ such that  $\a > \b$ for $\b \in \cB$. It remains to show that $R$ is preserved under the maps $\PHI_\a$. Suppose that $x \in R$ and $\a > \b$ for all $\b \in \cB$. We will show that $\PHI_\a(x) \in R$. To this end, it suffices to check that $\PHI_\a(x) \in R_\b$ for $\b \in \cB$, which amounts to $\PHI_\a(x) \in \tilde{S}$ and $\PHI_\b( \PHI_\a(x)) \in \tilde S$. The first condition follows immediately from $x \in R_\a$. Since $\PHI_\b \circ \PHI_\a = \PHI_{\a \cup \b}$, the latter follows from $x \in R_{\a \cup \b}$. Thus $P_\a(R)\subset R$ for $\a$ large enough. Passing yet again to an \IP\ subset, we ensure that the claim holds for all $\a$.
\end{proof}

\subsection*{Fractional parts of polynomials}

Our next aim is to show that the set $R$ constructed in Step \ref{prop:S:main-3} needs to be invariant under a wider range of operations. For this purpose, we study the possible algebraic relations between polynomials and their fractional parts.

\begin{proposition} \label{lem:variable-separation}
	Let $G/\Gamma$ be a nilmanifold equipped with a \Malcev\ basis $\cX$. Let $g,h\colon \ZZ \to G$ be polynomial sequences with values in $G^{\circ}$. 
	Suppose that $E$ is an \IP\ set and $\PHI \colon G^{\circ} \times G^{\circ} \to \RR$ is a polynomial map such that 
	\begin{equation}
	\label{eq:782-1}
	\PHI\bra{ g(n), \fp{h(n)} } = 0, \qquad n \in E.
	\end{equation}
	Then there exists an \IP\ set $E' \subset E$ such that
	\begin{equation}
	\label{eq:782-1a}
	\PHI\bra{g(m),\fp{h(n)} } = 0, \qquad m \in \ZZ,\ n \in E'.
	\end{equation}
\end{proposition}
\begin{remark}
	If $g$ is extended to a polynomial sequence $g\colon \RR \to G$ (defined in an obvious way) with values in $G^{\circ}$, then the same argument gives $\PHI\bra{g(m),\fp{h(n)} } = 0$ for $m \in \RR$, $n \in E'$.
\end{remark}
\begin{proof} We may assume that $g$ is defined on $\RR$. Consider the map $\PSI \colon \RR \times G^{\circ} \to \RR$ given by $\PSI(m,y) =\PHI(g(m),y)$. This is a polynomial map.
 Expand $\PSI$ as $ \PSI(x,y) = \sum_{k=0}^d x^k \PSI_k(y)$ for polynomials $\PSI_k \colon G^{\circ} \to \RR$. By the assumption, $Q(n,\fp{h(n)})=0$ for $n\in E$.

	Looking at the leading term in the equation 
	\begin{equation}
	\label{eq:782-2}
		\sum_{k = 0}^d n^k \PSI_k\bra{ \fp{h(n)}  } = 0, \quad n\in E
	\end{equation}
	and using the fact that $Q_k$ are bounded functions on the fundamental domain $D$ of $G/\Gamma$, we conclude that 
	\begin{equation}
	\label{eq:782-25}
	\lim_{E \ni n \to \infty} \PSI_d\bra{ \fp{h(n)}  } = 0.
	\end{equation}
	Take $(n_i)_{i\in \N}$ such that $n_\a \in E$ for $\a \in \cF$.	In particular, as a special case of \eqref{eq:782-25}, for any fixed $\b \in \cF$ we have
	\begin{equation}
	\label{eq:782-3}
	\iplim_{\a \in \cF} \PSI_d \bra{ \fp{h\bra{n_\a + n_\b}}}  = 0.		
	\end{equation}	
	By the \IP\ polynomial recurrence theorem for nilrotations (see \cite[Theorem D]{Leibman-2005b}), 	there exists an \IP\ ring $\cF_\b$ (which may be chosen to be contained in any previously specified \IP\ ring) such that 
	\begin{equation}
	\label{eq:782-4}\iplim_{\a \in \cF_\b} h(n_\a + n_\b) \Gamma = h(n_\b) \Gamma.
	\end{equation}

	We would like to take now the fractional parts in the preceding limit. A slight technical difficulty stems from the fact that the map $g \mapsto \fp{g}$ is discontinuous. To overcome this problem, let $\e > 0$ be sufficiently small so that if $x \in G$ lies in the fundamental domain, then $d(x, x \gamma) \geq \e$ for all $\gamma \in \Gamma\setminus\{e\}$. After replacing $E$ with an \IP\ subset $E'$ in the original statement, we may assume that all the points $\fp{h(n_{\a})}$ for $\a \in \cF$ lie within a ball of radius $\e/10$. This will be useful shortly.

	Applying \eqref{eq:782-4}, possibly after refining $\cF_\b$ further, we find that
	\begin{equation}
	\label{eq:782-5}\iplim_{\a \in \cF_\b} \fp{ h(n_\a + n_\b)} = 
	\fp{ h(n_\b) }\gamma(\b),	
	\end{equation}
	where $\gamma(\b) \in \Gamma$. Now, since two limit points of $\fp{h(n_\a)}$ cannot differ by a factor of $\gamma \in \Gamma \setminus \{e\}$, we conclude that 
		\begin{equation}
	\label{eq:782-6}\iplim_{\a \in \cF_\b} \fp{ h(n_\a + n_\b)} = \fp{ h(n_\b)}.
	\end{equation}
	
	 Because $\PSI_d$ is continuous, \eqref{eq:782-6} yields $\PSI_d \bra{ \fp{ h(n_\b)}} = 0$, where we recall that $\b$ was arbitrary. Reasoning inductively, we obtain similarly $\PSI_k \bra{ \fp{ h(n_\b)} } = 0$ for all $k$ and $\b \in \cF$. Consequently, $\PSI\bra{m, \fp{ h(n)}} = 0$ for all $m \in \ZZ$ and $n \in E'$, as desired.
\end{proof}

\begin{step}\label{prop:S:main-4} 	Suppose that Theorem \ref{thm:S:main-1} is false. Then there exists a nilmanifold $G/\mathord\Gamma$ with a \Malcev\ basis $\cX$  such that every $G_i/\mathord\Gamma_i$ is connected, where $\Gamma_i=\Gamma\cap G_i$, an element $g \in G$ acting ergodically by left translations, a nonempty algebraic subset $R \subset G^{\circ}$ with empty interior, and an \IP\ set $E$ such that $R$ is preserved under the operations
\begin{align*}
	x \mapsto x \fp{g^{n}}, \qquad n \in E.
\end{align*}  
\end{step}
\begin{proof}
	Assume the notation is as in Step \ref{prop:S:main-3}. 	Let us consider for $x \in R$ the set
	$$
		M_x = \set{ (m, n) \in \ZZ \times \ZZ }{ g^m x g^{-m} \fp{ g^n } \in R}.
	$$
	Since multiplication in $G^{\circ}$ is given by polynomial formulae and $R$ is an algebraic set, the condition $(m,n) \in M_x$ is equivalent to a system of polynomial equations in $g^m x g^{-m}$ and $\fp{ g^n }$. By Step \ref{prop:S:main-3}, we have $(n,n) \in M_x$ for all $n \in E$. We first apply Lemma \ref{lem:poly-seq-in-G-o} to replace the polynomial sequence $g^n$ by a polynomial sequence $h(n)$ with values in $G^{\circ}$ and $\fp{g^n}=\fp{h(n)}$. Then Lemma \ref{lem:variable-separation} shows that $M_x$ contains $\ZZ \times E'$ for an \IP\ set $E'$ (which may be chosen inside a given \IP\ subset of $E$).

 For $x \in R$ let us also consider the set 
	$$
		P_{x} = \set{ h \in G^{\circ} }{ g^m x g^{-m} h \in R \text{ for all } m \in \ZZ }.
	$$	
	and let $P = \bigcap_{ x \in R } P_{x}$. These sets are algebraic, and $\fp{g^n} \in P_x$ precisely when $\ZZ \times \{n\} \subset M_x$. By Hilbert's Basis Theorem, we have $P = \bigcap_{ x \in R_0 } P_{x}$ for a finite subset $R_0$ of $R$. 
	
	Iterating the argument above, we find an \IP\ set $E'$ such that $\ZZ \times E' \subset \bigcap_{x \in R_9} M_x$.	
 Hence $\fp{ g^{n} } \in P$ for $n\in E'$. In particular, $ R \fp{g^n} \subset R$ for $n \in E'$, and so $R$ is preserved by the map $x \mapsto x \fp{g^{n}}$.
\end{proof}

\subsection*{Group generated by fractional parts}

We are now ready to take the last step in the proof of Theorem \ref{thm:S:main-1}. Our strategy is to show that the set $R$ appearing in Step \ref{prop:S:main-4} in invariant under multiplication by all of $G^{\circ}$, which is clearly absurd. In order to facilitate this strategy, we need the following fact.

\begin{proposition}\label{lem:fp-closure}
	Let $X=G/\Gamma$ be a connected nilmanifold with a fixed  \Malcev\ basis and such that every $G_i/\mathord\Gamma_i$ is connected, where $\Gamma_i=\Gamma\cap G_i$. Let $g \colon \Z \to G$ be a polynomial sequence such that the sequence $(g(n)\Gamma)_{n\geq 0}$ is dense in $X$. Suppose that $H \subset G^{\circ}$ is a Lie subgroup such that $\fp{g(n)} \in H$ for infinitely many $n$. Then $H= G^{\circ}$. 
\end{proposition}

\begin{step}\label{prop:S:main-5}
	Theorem \ref{thm:S:main-1} holds true.
\end{step}
\begin{proof}[Proof (assuming Propositon \ref{lem:fp-closure})]
	Suppose Theorem \ref{thm:S:main-1} were false, and assume notation as in conclusion of Step \ref{prop:S:main-4}. Denote $$H = \set{ h \in G^{\circ}}{ R h \subset R} = \bigcap_{x \in R} \set{ h \in G^{\circ} }{ x h \in R}.$$
	
	It is immediate from the definition that $H$ is algebraic and closed under multiplication and taking inverses (by Lemma \ref{lem:invariant=>strongly-invariant}). Hence, $H$ is a Lie subgroup of $G^{\circ}$. Furthermore, $\fp{g^{n}} \in H$ for $n \in E$.
Applying Proposition \ref{lem:fp-closure}, we conclude that $H = G^{\circ}$. In particular, $R = G^{\circ}$, which is in contradiction with $\inter R = \emptyset$.
\end{proof}

\begin{proof}[Proof of Proposition \ref{lem:fp-closure}]

	Let $E$ be an infinite set with $\fp{g(n)} \in H$ for $n \in E$. 
	
	By Lemma \ref{lem:poly-seq-in-G-o}, there exists a polynomial sequence $h\colon \Z\to G$ with values in $G^{\circ}$ such that $\fp{g(n)}=\fp{h(n)}$. Write $$h(n)=	 g_0^{\binom{n}{0}} g_1^{\binom{n}{1}} g_2^{\binom{n}{2}} \dots g_d^{\binom{n}{d}},$$ where $g_i\in G$. Since $h(n)$ takes values in $G^{\circ}$, we see that $g_i$ lie in $G^{\circ}$. 
	
	In order to prove the claim, it is sufficient by Lemma \ref{lem:G2-not-complemented} to show that $H(G^{\circ})_2 = G^{\circ}$.	The quotient $G^{\circ}/(G^{\circ})_2 $ can be identified with $\RR^d$ in such a way that $\Gamma\cap G^{\circ}$ maps to $\ZZ^d$. Let $\pi \colon G^{\circ} \to \RR^d$ be the corresponding projection. Put $\theta(n) = \pi( h(n) )$, $\xi(n) = \pi( \fp{h(n)} ) = \fp{ \theta(n) }$ and $V = \pi(H) < \RR^d$. From the form of $h(n)$, we see that $\theta(n)$ is a polynomial in $n$, $\xi(n)$ is bounded, and $\xi(n) \in V$ for $n\in E$. We need to show that $V = \RR^d$.

	Suppose for the sake of contradiction that $V \neq \RR^d$. Write $\theta(n)=(\theta_1(n),\ldots,\theta_d(n))$. Then there exists a non-trivial linear relation
	\begin{equation}
	\label{eq:324-1}	
		\sum_{i=1}^d \lambda_i \fp{\theta_i(n)}  = 0, \qquad n \in E,
	\end{equation}
	where $\lambda_i \in \RR$ are not all zero. This can be rewritten as
	\begin{equation}
	\label{eq:324-2}	
		\sum_{i=1}^d \lambda_i \ip{\theta_i(n)} = \sum_{i=1}^d \lambda_i \theta_i(n), \qquad n \in E.
	\end{equation}
	The maps $\theta_i \colon \Z \to \R$ are polynomials, so writing 
	$$ - \sum_{i=1}^d \lambda_i \theta_i(n) = \sum_{j=0}^D \kappa_j n^j,$$
	we obtain
	\begin{equation}
	\label{eq:324-3}
		\sum_{i=1}^d \lambda_i \ip{\theta_i(n)} + \sum_{j=0}^D \kappa_j n^j = 0 , \qquad n \in E.
	\end{equation}
	This amounts to saying that there is a non-trivial linear relation between the vectors $(\ip{\theta_i(n)})_{n \in E}$ for $i = 1,\dots,d$ and $(n^j)_{n \in E}$ for $j = 0,1,\dots,D$. Because these vectors are integer-valued, if such a relation exists, there is also such a relation with integer coefficients. Take one such relation
	\begin{equation}
	\label{eq:324-7}
		\sum_{i=1}^d l_i \ip{\theta_i(n)} + \sum_{j=0}^D k_j n^j = 0 , \qquad n \in E,
	\end{equation}
	where $l_i \in \ZZ$ and $k_j \in \ZZ$ for all $i,j$, and not all of $l_i,k_j$ are zero. 		Dropping the integer parts in \eqref{eq:324-7} at the cost of introducing a bounded error, we get
	\begin{equation}
	\label{eq:324-4}
		\sum_{i=1}^d l_i {\theta_i(n)} + \sum_{j=0}^D k_j n^j = O(1), \qquad n \in E.
	\end{equation}
	This is only possible if the left hand side of \eqref{eq:324-4} is constant (as a polynomial, hence for all $n \in \ZZ$).	It follows that $\theta(n) \bmod{\ZZ^d}$ for $n \in \ZZ$ takes values in the proper subtorus $\set{x \in \RR^d/\ZZ^d}{ \sum l_i x_i = c }$, which contradicts the fact that the sequence $(g(n)\Gamma)_{n\geq 0}$ is dense in $G/\Gamma$.\end{proof}

\section{Small fractional parts}\label{sec:Examples}

In this section, we study the sequences of the form
\begin{equation}
	g(n) = 
	\begin{cases}
	1, & \text{ if }  0 < q(n) < \e(n), \\
	0, & \text{ otherwise,} 
	\end{cases} 
	\label{eq:def-standard-sparse}
\end{equation}
 where $q(n) \geq 0$ is a generalised polynomial and $\e(n) \to 0$, which include the most natural examples of sparse generalised polynomials. The most frequent application is when $q(n) = \fpa{r(n)}$ or $q(n) = \fp{r(n)}$  for an (unbounded) generalised polynomial $r(n)$. In this section, it will be more convenient to regard generalised polynomials as functions on $\N_0$. This causes no problems since we can always extend them to $\Z$.
 
 \begin{lemma}\label{lem:gen-poly-set-of-zeros}
	If $h$ is an (unbounded) generalised polynomial, then 
	$$g(n) =
	\begin{cases}
	1, & \text{if } h(n) = 0, \\
	0, & \text{otherwise,}
	\end{cases}
	$$
	is a generalised polynomial. Likewise, for any $a < b$, 
	$$g'(n) =
	\begin{cases}
	1, & \text{if }  a \leq h(n) < b, \\
	0, & \text{otherwise,}
	\end{cases}
	$$
	is a generalised polynomial.	
\end{lemma}
\begin{proof}
 Because $h(n)$ takes countably many values for $n \in \NN_0$, there exists $\theta \in \RR$ such that $\theta h(n) \in \RR \setminus \QQ$ unless $h(n) = 0$. Now, a short computation verifies that 
 $$g(n) = \lfloor 1 - \fp{\theta h(n)}\rfloor.$$
 For the latter claim, after rescaling we may assume that $a = 0$ and $b = 1$. Then $g'(n) = \ifbra{\floor{ h(n)} = 0}$, and we may apply the construction above to $h'(n) = \floor{ h(n)}$.
\end{proof}

\begin{example}\label{ex:lin-rec-is-gp}
	Take $a \in \NN$ and let $(n_i)_{i \geq 0}$ be the sequence given by $n_0 = 0,\ n_1 = 1$ and $n_{i+2} = a n_{i+1} + n_i$. Then $E = \set{n_i}{i \in \NN_0}$ is generalised polynomial. For $a = 1$, this is the set of Fibonacci numbers.
\end{example}
\begin{proof}
Let $\a \in \RR$ be the real number with continued fraction expansion $\a = [a;a,a,\dots]$, i.e.
$$
	\a = a+\cfrac{1}{a+\cfrac{1}{a+\cdots}} = \frac{a + \sqrt{a^2+4}}{2}.
$$
Let $E' = \set{n \in \NN}{ \norm{n\a} < 1/2n }$. Note that $E'$ is generalised polynomial by Lemma \ref{lem:gen-poly-set-of-zeros}, so it will suffice to show that the symmetric difference $E \triangle E'$ is finite.

By a classical theorem of Legendre (see e.g.\ \cite[Thm.\ 5.1]{Khintchine-book}), we have $E' \subset E$. Conversely, using a well-known formula for the error term in the continued fraction approximations (see e.g.\ \cite[Thm.\ 3.1]{Khintchine-book}), we find that
$$
	n_i \norm{n_i \a} = n_i^2 \sum_{l=0}^\infty \frac{(-1)^{l}}{n_{i+l} n_{i+l+1}} \xrightarrow[i \to \infty]{} \sum_{l=0}^\infty \frac{(-1)^l}{\a^{2l+1}} = \frac{1}{\a + 1/\a} = \frac{1}{\sqrt{a^2 + 4}},
$$
because $n_{i+1}/n_i \to \alpha$ as $i \to \infty$. Since $1/\sqrt{a^2 + 4} < 1/2$, for sufficiently large $i$ we have $\norm{n_i \a} <  1/2 n_i$, whence $n_i \in E'$. 
\end{proof}

\begin{remark}\label{lawnmower}
	In fact, this result holds for any choice of $n_0, n_1\in\Z$. This follows from Proposition \ref{przepraszamzetakpozno} below. A similar argument works also for sequences $(n_i)_{i \geq 0}$ given by the recurrence $n_{i+2} = a_{i+2} n_{i+1} + n_i$, where $a_i \geq 2$ is a bounded sequence of integers. Then $n_i$ is the sequence of denominators in the best rational approximations of a badly approximable real number $\alpha=[0;a_2,a_3,\ldots]$.
\end{remark}

\begin{proposition}\label{prop:special-gp-is-gp}
	Suppose $q(n)\geq 0$ and $p(n)>0$ are generalised polynomials, $b<0$ is a rational number, $\lim_{n \to \infty} p(n) = \infty$, and $\e(n) = p(n)^b$. Then $f(n)$ given by \eqref{eq:def-standard-sparse} is a sparse generalised polynomial.
\end{proposition}
\begin{proof}
	That $f(n)$ is a generalised polynomial follows from Lemma \ref{lem:gen-poly-set-of-zeros}; that it is sparse follows by a standard argument from Theorem \ref{BLnilgenpolythm} via Corollary \ref{cor:density-uniform}. \end{proof}

In dealing with polynomials, it is often useful to exploit the fact that sufficiently high discrete derivatives vanish. More precisely, if $q$ is a polynomial of degree at most $d-1$, then for any $n_1,\dots,n_d$ we have
$$
	\sum_{\a \subseteq \{1,\ldots,d\}} (-1)^{\abs{\a}} q\bra{ n_\a } = 0,
$$
	with the usual convention $n_\alpha = \sum_{i \in \alpha} n_i$.
	
	A similar relation holds also for generalised polynomials. Such a relation  is stated in \cite[Theorem 2.42]{BergelsonMcCutcheon-2010} in terms of limits along ultrafilters. Below, we give a statement purely in terms of \IP\ limits, which is easily deduced from \cite{BergelsonMcCutcheon-2010} or proven directly.

\begin{theorem}\label{thm:discrete-derivative-vanishes}
	For any generalised polynomial $q$ and an \IP\ ring $\cG_0$, there exist $\lambda \in \RR$, $d \in \NN_0$ (equal to the degree of $q$), and families of  \IP\ subrings $\cG_1$, $\cG_2(\b_1), \ldots,  \cG_d(\b_1,\dots,\b_{d-1})$ of $\cG_0$, $\b_i \in \cF$, such that 
\begin{equation}
\label{eq:763}
	\sum_{\emptyset \neq \a \subseteq \{1,\ldots,d\}} (-1)^{\abs{\a}} q\bra{ \sum_{i \in \a} n_{\b_i} } = - \lambda
\end{equation}	
	whenever $\b_1 \in \cG_1,\ \b_2 \in \cG_2(\b_1),\ \dots,\ \b_d \in \cG_d(\b_1,\dots, \b_{d-1})$.
	\end{theorem}

	Using the above relation, we are able to prove a special case of Theorem \ref{thm:main-A} with negligible effort.
\begin{proposition}\label{prop:IP-auto-vs-special-gp}
	Let $q(n) \geq 0$ be a generalised polynomial, $\e(n) \geq 0$ with $\e(n) \to 0$ as $n \to \infty$, and let $g$ be given by \eqref{eq:def-standard-sparse}. Then there exists no \IP\ set $E$ such that $f(n) = 1$ for $n \in E$.
\end{proposition}

\begin{proof}[Proof of Proposition \ref{prop:IP-auto-vs-special-gp}]
	Suppose that $(n_\a)_{\a \in \cF} $ were an \IP\ set such that $ f(n_\a) = 1$ for $\a \in \cF$.
	Apply Theorem \ref{thm:discrete-derivative-vanishes} to find $\lambda \in \RR$ and families of \IP\ rings $\cG_j(\b_1,\dots,\b_{j-1})$ such that
\begin{equation}
	\sum_{\emptyset \neq \a \subseteq \{1,\ldots,d\}} (-1)^{\abs{\a}} q\bra{ \sum_{i \in \a} n_{\b_i} } = - \lambda
\end{equation}	
	for $\b_1 \in \cG_1,\ \b_2 \in \cG_2(\b_1),\ \dots,\ \b_d \in \cG_d(\b_1,\dots, \b_{d-1})$. Passing to the \IP\ limit with $\beta_d, \ldots, \beta_2$ (in this order), we obtain $-q(n_{\beta_1})=-\lambda$. This proves that $$0<\lambda<\varepsilon(n_{\b_1}),$$ which gives a contradiction.
\end{proof}

\begin{corollary}
	If $p(x) \in \RR[x]$ has least one irrational coefficient and $\e(n) \to 0$, then the set $\set{n \in \ZZ}{ \fpa{p(n)} < \e(n) }$ does not contain an \IP-set.
\end{corollary}

We move on to an example of a fairly explicit class of sparse generalised  polynomial sequences. We later discuss how the argument adapts to other sequences of the form \eqref{eq:def-standard-sparse}.

\begin{proposition}\label{ex:Heisenberg} For any sufficiently small $c > 0$ the following holds: Let $1,\a,\b \in \RR$ be algebraic numbers linearly independent over $\QQ$, and let $\e(n) \to 0$ be a rational power of a generalised polynomial with $\e(n) \gg n^{-c}$. Then the sequence given by
	$$f(n) = 
	\begin{cases}
	1, & \text{ if }  \fpa{ n\a \floor{n \b} } < \e(n), \\
	0, & \text{ otherwise} 
	\end{cases} 
	$$
	is a sparse generalised polynomial with $\sum_{n = 0}^{N-1} f(n) \gg N^{1-c}$. (In particular, $f$ is not eventually $0$.) Moreover, for any $a \in \NN, b \in \ZZ$, we also have $\sum_{n = 0}^{N-1} f(an + b) \gg N^{1-c}$ (where the implicit constant depends on $a$ and $b$).
\end{proposition}

A key tool which we will use is the quantitative equidistribution theorem for nilrotations, whose special case cited below we use as a black box.

\begin{theorem}[Green-Tao, \cite{GreenTao2012}]\label{thm:GreenTao}
	Let $g(n)$ be a polynomial sequence on a nilmanifold $G/\Gamma$ where $G$ is connected. Fix a \Malcev\ basis $\cX$ of $G$. Then there exists a constant $A > 0$ such that for any $0 < \delta < 1/2$ and any $N$, one of the following conditions holds:
	\begin{enumerate}
	\item\label{item:GT:1} for each $x \in G/\Gamma$, there exist $\delta N$ values of $n \in [N]$ with $d_{\cX}( g(n)\Gamma, x) \leq \delta$; 	\item\label{item:GT:2} there exists a horizontal character $\eta$ with $\norm{\eta} \ll \delta^{-A}$ such that 
	$$
		\max_{0 \leq n < N} \fpa{ \eta( g(n+1)) - \eta(g(n)) } \ll \delta^{-A} N^{-1}.
	$$
	\end{enumerate}
\end{theorem}

For linear orbits $g(n) = a^n$, $a \in G$, the condition \eqref{item:GT:2} in Theorem \ref{thm:GreenTao} 
 asserts that the projection of $a$ to the torus $G/G_2\Gamma$ satisfies an approximate linear relation over $\ZZ$. A convenient criterion which can be used to rule that out is provided by the following classical result of Schmidt. As before, we only cite the special case which we shall use.

\begin{theorem}[Schmidt \cite{Schmidt1972}]\label{thm:Schmidt}
	Let $\alpha_1,\dots,\alpha_n$ be $n$ algebraic numbers linearly independent over $\QQ$, and let $\e > 0$. Then for $k \in \ZZ^n$ we have
	$$
		\norm{ \sum_{i=1}^n k_i \alpha_i } \gg \norm{k}^{-n-\e}
	$$ 
	where the implicit constant depends on $\alpha_1,\dots,\alpha_n$ and $\e$.
\end{theorem}

\begin{proof}[Proof of Proposition \ref{ex:Heisenberg}]
	
	It is clear by Proposition \ref{prop:special-gp-is-gp} that $f(n)$ is a sparse generalised polynomial, what remains to be proved is the bound on its growth. We only deal with $\sum_{n=0}^{N-1} f(n)$, the estimation of $\sum_{n=0}^{N-1} f(an+b)$ is analogous. To this end, we may drop the assumptions that $\e(n)$ is a rational power of a generalised polynomial and assume that $\e(N) \sim N^{-c}$.
	
	We can explicitly describe a nilmanifold on which $f(n)$ can be realised in a way analogous to Theorem \ref{thm:BergelsonLeibman}. Denote 
	$$[x,y,z] \equiv \begin{bmatrix}
		1 & x & z \\ 
		0 & 1 & y \\
		0 & 0 & 1 
	\end{bmatrix},$$
and define 
	$$
		G = \set{ [x,y,z] }{x, y, z \in \RR},\qquad 
		\Gamma \set{ [x,y,z] }{x, y, z \in \ZZ}.	
	$$
	Then, $G/\Gamma$ is the Heisenberg nilmanifold with a natural choice of a \Malcev\ basis $[1,0,0],\ [0,1,0],\ [0,0,1]$.
	
	The horizontal characters on $G/\Gamma$ take the form $\eta_k([x,y,z]) = k_1 x + k_2 y $ with $k = (k_1,k_2) \in \ZZ^2$ and $\norm{\eta_k} = \norm{k}_2$. Let $g$ be the polynomial sequence
	$$g(n) = [- n \a, n \b, 0] = [-\a,\b,0]^n [0,0,\a\b]^{\binom{n}{2}},$$ so that $\fp{g(n)} = [ \fp{-n\a}, \fp{n\b}, \fp{ n\a \floor{n\b} }]$. 	
	Let $F \colon G/\Gamma \to [0,1)$ be the map $F([x,y,z]) = \fpa{z}$ for $x,y,z \in [0,1)^3$. It follows from the way the metric on $G/\Gamma$ is defined that $F$ is Lipschitz. \comment{(Admittedly, not super clear. However, I think I would like to leave it like that, because 1) it is intuitively clear that it should be so 2) this is an example, so we can afford to be slightly imprecise 3) the alternative is to go into the details of how the metric on $G/\Gamma$ is defined, and it's neither enlightening nor pleasant...)} Then $f(n) = \ifbra{ F(g(n)) < \e(n) }$.

	Let $N$ be a sufficiently large integer, and put $z = [0,0,0]\Gamma$. There exists $\rho_N \gg \e(N)$ such that $F(B(z,\rho_N)) \subset [0,\e(N))$. Hence, if $d(g(n)\Gamma, z) < \rho_N$ for some $n \in [N]$, then $f(n) = 1$. If $f(n) = 1$ for at least $ N\rho_N \gg N^{1-c}$ values of $n \in [N]$, we are done. Suppose this is not the case. It now follows from Theorem \ref{thm:GreenTao} that there exists a horizontal character $\eta$ with $\norm{\eta} \ll \rho^{-A}_N \ll N^{Ac}$ such that 
	  \begin{equation}
		  \label{eq:255}	  
	  	\max_{0 \leq n < N} \fpa{ \eta( g(n+1) ) - \eta(g(n)) } \ll N^{-1+Ac},	  
	  \end{equation}
	  where $A$ is an absolute constant. Picking $l \in \ZZ^2$ so that $\eta = \eta_l$, we conclude from \eqref{eq:255} and Theorem \ref{thm:Schmidt} that 
	  \begin{equation}
	    \label{eq:256}	  
	  	N^{-3Ac} \ll \norm{l}^{-3} \ll \fpa{ - l_1 \a + l_2 \b } \ll N^{-1+Ac}.
	  \end{equation}
As long as $c < 1/(4A)$, we get a contradiction.
\end{proof}

\begin{remark}\label{remark:Heisenberg}
	The above argument is not specific to the sequence $\ifbra{ n\a \floor{n\b} < \e(n)}$; in fact, the term $n\a \floor{n\b}$ could be replaced by a fairly arbitrary generalised polynomial $q(n)$. The difficulty lies in the need to impose suitable ``quantative irrationality'' conditions of $q(n)$.
	
	In the situation of Example \ref{ex:Heisenberg}, this was easily accomplished since we had a direct access to the representation of $q(n)$ in terms of an orbit on a nilmanifold. For more general sequences, one needs to resort to the  construction in \cite{BergelsonLeibman2007}.
\end{remark}

\subsection*{Very sparse sequences}

It is difficult to say anything substantial about generalised polynomial sequences which are extremely sparse. Indeed, as we presently show, any sufficiently sparse set is generalised polynomial.

We will need the following elementary lemma. We thank Aled Walker for simplifying our argument.
\begin{lemma}\label{lem:coprimes-in-intervals}
	For any $\delta > 0$, there exists $q_0 = q_0(\delta)$ such that for all $q \geq q_0$ and all $x\geq 0$, the interval $[x,x+q^\delta)$ contains an integer $n$ coprime to $q$. 
\end{lemma}
\begin{proof}
	Using an inclusion-exclusion argument, we find that
	\begin{align}
		\sum_{\substack{ x \leq n < x+q^\delta \\ (n,q) = 1} } 1
		= \sum_{d | q} \left(\mu(d) \frac{q^\delta}{d} + O(1)\right)
		= q^\delta \sum_{d|q} \frac{\mu(d)}{d} + O(\tau(q)) 
		= q^\delta \frac{\varphi(q)}{q} + O(\tau(q)), \label{eq:637}
	\end{align}
	where $\mu$ is the M\"{o}bius function, $\varphi$ is the Euler totient function, and $\tau$ is the number of divisors function. Using the standard estimates $\varphi(q) \gg q^{1-\delta/3}$ and $\tau(q) \ll q^{\delta/3}$, we conclude that the expression in \eqref{eq:637} is positive for $q$ large enough.
\end{proof}

\begin{proposition}\label{prop:nicely-sparse=>gen-poly}
Let $C$, $D$ be integers satisfying $5 \leq C < D \leq \frac{1}{2} (C-1)^2$. Suppose that the sequence $(n_i)_{i \geq 0}$ of integers satisfies $n_{0} \geq 2$ and $n_{i}^D < n_{i+1} < n_{i}^{2D}$ for $i \geq 0$. Then there exists $\a \in (0,1)$ such that the symmetric difference of the sets $E = \set{n_i}{ i \geq 0}$ and $$E' = \set{n \in \NN_0}{ \frac{1}{4} n^{-C+1} \leq \fpa{n \a} \leq \frac{1}{2} n^{-C+1}}$$  is finite.
\end{proposition}
\begin{proof}
	We inductively construct a sequence of integers $(m_i)_{i \geq 0}$ such that the  intervals $I_i = [\frac{m_i}{n_i} + \frac{1}{4} n_i^{-C}, \frac{m_i}{n_i} + \frac{1}{2} n_i^{-C}]$ are nested (i.e.\ $I_{i+1} \subset I_{i}$) and moreover $(m_i,n_i) = 1$ for sufficiently large $i$. Choose $m_0 = 0$. If $m_i$ has been constructed, then the possible choices of $m_{i+1}$ such that the nesting condition is satisfied form an interval of length $\gg n_{i+1}/n_{i}^C \gg n_{i+1}^{1-C/D}$. Because $C<D$, by Lemma \ref{lem:coprimes-in-intervals}, for sufficiently large $i$ there is a choice of $m_{i+1}$ with $(m_{i+1},n_{i+1}) = 1$.
	
	Once the sequence $(m_i)_{i \geq 0}$ has been constructed, pick $\a$ so that $\{\a\} = \bigcap_{i \geq 0} I_{i}$. By construction, for this choice of $\a$ we have $E \subseteq E'$. Let $\a = [0;a_1,a_2,\dots]$ be the continued fraction expansion of $\a$, and let $\frac{p_j}{q_j} = [0;a_1,\dots,a_j]$ be the corresponding convergents. We recall an estimate for the error term of the approximation: $\frac{1}{2} \frac{1}{q_j q_{j+1}} \leq | \a - \frac{p_j}{q_j}| \leq \frac{1}{q_j q_{j+1}}$. Since $C \geq 2$, the classical theorem of Legendre implies that any approximation $\frac{m}{n}$ with $(m,n) = 1$ and $| \a - \frac{m}{n} |\leq \frac{1}{2} n^{-C}$ is equal to  $\frac{p_j}{q_j}$ for some $j$. Furthermore, in this case $q_{j+1} \geq q_j^{C-1}$. In particular, for each $i$, there is $j(i)$ such that $\frac{m_i}{n_i} = \frac{p_{j(i)}}{q_{j(i)}}$ and $q_{j(i)+1} \geq q_{j(i)}^{C-1}$. 
	
	Suppose now that $n \in E' \setminus E$. Then there exists $m$ (not necessarily coprime to $n$) with $\frac{1}{4} n^{-C} \leq | \a - \frac{m}{n} | \leq \frac{1}{2} n^{-C}$, whence $\frac{m}{n} = \frac{p_j}{q_j}$ for some $j$. If $\frac{m}{n}$ were equal to $\frac{m_i}{n_i}$ for some $i$ sufficiently large that $(n_i,m_i) = 1$, then contradiction would follow immediately: $n^{-C} \leq \frac{1}{4} n_i^{-C} \leq | \a - \frac{m_i}{n_i} | = | \a - \frac{m}{n} | \leq  \frac{1}{2} n^{-C}$. The cases when $q_j$ is bounded account for finitely many possible values of $n$. Hence, we may assume that $j$ lies strictly between $j(i)$ and $j(i+1)$ for some $i$ which is sufficiently large so that $(m_{k},n_{k}) = 1$ for $k \geq i$. We then have the inequalities $$q_{j} \geq q_{j(i)+1} \geq q_{j(i)}^{C-1}=n_i^{C-1}$$ and $$n_{i+1} = q_{j(i+1)} \geq q_{j+1} \geq q_{j}^{C-1} \geq n_i^{(C-1)^2},$$ which imply that $n_{i}^{2D} > n_{i+1} \geq n_{i}^{(C-1)^2}$, contradicting $2D \leq (C-1)^2$.
\end{proof}	

We are now ready to prove another one of our main results, whose formulation we reiterate.
\setcounter{alphatheorem}{2}
\begin{alphatheorem}\label{cor:very-sparse=>gen-poly}
	There exists a constant $c > 0$ such that for any sequence $(n_i)_{i \geq 0}$ with $n_0 \geq 2$ and $n_{i+1} \geq n_{i}^c$ for all $i \geq 0$, the set $E = \set{n_i}{i \geq 0}$ is generalised polynomial.
\end{alphatheorem}
\begin{proof}
	We first construct a sequence  $(n_j')_{j \geq 0}$ of integers with $n_j'^6 < n_{j+1}' < n_{j}'^{12}$ for sufficiently large $j$ and such that $(n_i)_{i \geq 0}$ is its subsequence. We construct $n_j'$ inductively, keeping track of $j(i)$ such that $n_{j(i)}'=n_i$. For $i = 0$, pick $j(0) = 0$ and $n'_0 = n_0$. If $j(i)$ and $(n_j)_{0 \leq j \leq j(i)}$ have been constructed, we define $j(i+1) = j(i) + l$ and $n_{j(i)+k}' = \floor{ \exp( (\log n_{i+1})^{k/l} (\log n_{i})^{1-k/l}  ) }$ for $0 \leq k \leq l$, where $l$ remains to be chosen. This ensures that $\frac{\log n_{j(i)+k+1}'}{\log n_{j(i)+k}'} = \bra{ \frac{\log n_{i+1}}{\log n_i} }^{1/l} + o(1)$ as $i \to \infty$. Putting $A = \frac{\log n_{i+1}}{\log n_i} \geq c$, and letting $l$ be any integer in the interval $( \frac{\log A}{\log 11}, \frac{\log A}{\log 7} )$, we obtain (for sufficiently large $i$) the inequality  $6 < \frac{\log n_{j(i)+k+1}'}{\log n_{j(i)+k}'} < 12$. This interval indeed contains at least one integer, provided that $c$ is large enough.
	
	Applying Proposition \ref{prop:nicely-sparse=>gen-poly} with $C = 5, D = 6$, we conclude that the set $E' = \set{n_j'}{j \geq 0}$ is generalised polynomial. By the same argument, the set $E'' = \set{n_j''}{j \geq 0}$ is also generalised polynomial, where $n''_{j(i)} = n_i$ and $n''_j = n'_j + 1$ if $j \neq j(i)$ for all $i$. It follows that the set $E' \cap E'' = E$ is generalised polynomial, as required.	\end{proof}

\begin{remark} The constant $c$ in the previous proposition can be effectively computed. A slight modification of the method above gives $c= 617$.

\end{remark} 
\section{Exponential sequences}\label{sec:Exponential}
 In order to construct a generalised polynomial that takes value zero exactly on the set of Fibonacci numbers, we used the theory of continued fractions. In this section, we will do this more generally for certain linear recurrence sequences of degree $2$ and $3$. We need to begin, however, by the following basic result.

 \begin{proposition}\label{przepraszamzetakpozno} Let $\beta$ be a Pisot number (or a Pisot-Vijayaraghavan number), i.e.\ a real algebraic integer $\beta>1$ such that all its Galois conjugates have absolute value smaller than one. Let $P(x)=x^d-c_1x^{d-1}-\ldots-c_d$ be the minimal polynomial of $\beta$. Let $(R_i)_{i\geq 0}$ be a sequence of integers satisfying the linear recurrence relation $$R_{i+d}=c_1R_{i+d-1}+\ldots+c_dR_i, \quad i\geq 0.$$  Assume that $(R_i)_{i\geq 0}$ is not identically zero. Then the following conditions are equivalent: \begin{enumerate} \item The set $\{R_i\mid i \in \N_0\}$ is generalised polynomial. \item The set $
 \{\lbr \beta^i \rbr \mid i \in \N_0\}$ is generalised polynomial. \end{enumerate} In particular, the question whether the set $\{R_i\mid i\in \N_0\}$ is generalised polynomial depends only on the recurrence relation and not on the initial values of $(R_i)_{i\geq 0}$.\end{proposition}
 
 \begin{proof} From the form of the linear recurrence relation for $(R_i)_{i\geq 0},$ we see that $R_i=u\beta^i + o(1)$ for some $u\in \R$. Note that $u \neq 0$. In fact, otherwise $R_i=o(1)$ and since $R_i$ takes integer values we see that $R_i=0$ for large $i$. It follows that $R_i$ is identically zero (note that $c_d\neq 0$). 
 
 Since $\beta$ is an algebraic integer, we know that $\sum_{\beta'} (\beta')^i$ is an integer, the sum being taken over all the Galois conjugates $\beta'$ of $\beta$. Since $|\beta'|<1$ for $\beta'\neq \beta$, we see that $\beta^i = \lbr \beta^i \rbr + o(1).$ It follows that $$R_i = u \lbr \beta^i \rbr + o(1)$$ and the claim follows immediately from the following lemma. \end{proof}
 \begin{lemma} Let $u\in \R$, $u\neq 0$, and let  $(R_i)_{i\geq 0} $ and $(S_i)_{i\geq 0} $ be integer-valued sequences such that $$R_i=uS_i+o(1).$$ Then the set  $\{R_i\mid i\in \N_0\}$ is generalised polynomial if and only if the set  $\{S_i\mid i\in \N_0\}$ is generalised polynomial.
 \end{lemma}
 
 \begin{proof} It follows from the relation $R_i=uS_i+o(1)$ that for sufficiently large $i$, an integer $m$ is of the form $m=S_i$ if and only if $\lbr u m \rbr = R_i$ and $\fpa{um}<|u|/2$. We immediately see from this that if the set $\{R_i\mid i\in \N_0\}$ is generalised polynomial, then so is  $\{S_i\mid i\in \N_0\}$. Since the argument is symmetric with respect to $(R_i)$ and $(S_i)$, we conclude the claim.\end{proof}
 
 \subsection*{Quadratic Pisot numbers}
 
 In this subsection we will generalise Example \ref{ex:lin-rec-is-gp}, where we have shown that the set of Fibonacci numbers in generalised polynomial.
 
 \begin{proposition}\label{thm:quadraticPisotisGP} Let $\beta$ be a quadratic Pisot unit, i.e.\ a quadratic Pisot number of (field) norm $\pm 1$.  Then the set $\{\lbr \beta^i \rbr \mid i\in \N_0\}$ is generalised polynomial.\end{proposition}

 \begin{proof} Let $\beta'$ be the Galois conjugate of $\beta$. Since $\beta$ is a quadratic Pisot unit, it satisfies the equation $\beta^2-a\beta-1=0$ for $a\geq 1$ (if the norm of $\beta$ is $-1$) or the equation $\beta^2-a\beta+1=0$ for $a\geq 3$ (if the norm of $\beta$ is $1$). In Example \ref{ex:lin-rec-is-gp}, we have constructed a sequence $(n_i)$ satisfying the recurrence $n_{i+2}=a n_{i+1} - n_i$ for $ i\geq 0$ and such that its set of values is generalised polynomial. Hence in this case the claim follows from Proposition \ref{przepraszamzetakpozno}.
 
 If the norm of $\beta$ is $1$, we need to argue a bit more carefully. The continued fractions exapansion of $\beta$ is $\beta = [a-1,\overline{1,a-2}]$. Let $(q_i)_{i\geq 0}$ be the sequence of denominators of convergents of this continued fraction. We get the recurrence relations $q_0=q_1=1$ and \begin{align*} q_{2i+2}&=(a-2)q_{2i+1}+q_{2i},  \\ q_{2i+3}&=q_{2i+2}+q_{2i+1}.\end{align*} By induction, we get from this the relations \begin{align*} q_{2i+4}&=aq_{2i+2}+q_{2i},  \\ q_{2i+5}&=aq_{2i+3}+q_{2i+1}.\end{align*}
 
  We solve this linear recurrence equations to get that $$q_{2i}=u_1\beta^i+u_2 (\beta')^i, \quad q_{2i+1}=v_1\beta^i+v_2 (\beta')^i, \quad i\geq 0$$ for some $u_1,u_2,v_1,v_2\in \R$. Since $0<\beta'<1$, we get $$q_{2i+1}= \lbr w q_{2i} \rbr, \quad q_{2i+2}= \lbr (\beta/w) q_{2i+1} \rbr$$ with $w=v_1/u_1$ for large $i$. Computing the initial values, we check that $w\neq \sqrt{\beta}$ if $a \geq 4$.
 
 Consider the set $E'=\set{n\in \N}{\norm{n\beta}<1/(2n)}.$ As in Example \ref{ex:lin-rec-is-gp}, we argue that $E'$ is generalised polynomial and that by Legendre's Theorem we have $E' \subset \set{q_i}{i\in \N_0}$. On the other hand, by the well-known error formula, we get $$\norm{q_{2i+1} \beta} < \frac{1}{q_{2i+2}} \leq \frac{1}{(a-2)q_{2i+1}}.$$ Let us assume that $a\geq 4$. Then we see that  $\set{q_{2i+1}}{i\in \N_0} \subset E'$. Since $q_{2i+1}= \lbr w q_{2i} \rbr$ and $q_{2i+2}= \lbr (\beta/w) q_{2i+1} \rbr$, $w\neq \beta/w$, we see that the set $$E=\set{n\in \N}{n\in E' \text{ and } \lbr w n \rbr \notin E'}$$ is equal to $\set{q_{2i+1}}{i\in \N_0}$ up to a finite set. Thus the set $\set{q_{2i+1}}{i\in \N_0}$ is generalised polynomial and the claim again follows from Proposition \ref{przepraszamzetakpozno}.
 
 It remains to treat the case $a=3$. In this case $\beta^2=3\beta+1$, and hence $(\beta^2)^2=7\beta^2 +1$. Applying the previous case to $\beta^2$, we see that the set $\{\lbr \beta^{2i} \rbr \mid i\in \N_0\}$ is generalised polynomial. By Proposition \ref{przepraszamzetakpozno}, so is the set $\{\lbr \beta^{2i+1} \rbr \mid i\in \N_0\}$ (since $\lbr \beta^{2i+1} \rbr$ and $\lbr \beta^{2i} \rbr$ satisfy for large $i$ the same linear recurrence). Thus, our claim holds in this case as well. \end{proof}

 \subsection*{Cubic Pisot numbers}
 
In order to replace the Fibonacci numbers by a linear recurrence sequence of higher degree, we need to use higher dimensional diophantine approximation. We briefly review the topic below. Let $\theta$ be a point in $\R^d$, $d\geq 1$, and let $N$ denote a norm on $\R^d$. Let $$N_0(\theta)=\inf_{p\in \Z^d} N(\theta-p)$$ denote the distance from $\theta$ to a nearest lattice point. We say that an integer $q\geq 1$ is a \emph{best approximation} of $\theta$ with respect to the norm $N$ if $N_0(q\theta)<N_0(k\theta)$ for $1\leq k\leq q-1$. If $\theta \notin \Q^d$, there are infinitely many best approximations of $\theta$. We order them in an increasing sequence $(q_i)_{i\geq 0}$, $1=q_0<q_1<\ldots.$ We call this sequence the \emph{sequence of best approximations}. For $d=1$, the sequence is well known as the sequence of denominators of convergents of the continued fraction expansion of $\theta$ (in the latter sequence the first term possibly appears twice). For $d\geq 2$ and general $\theta\in \R^d$ there is no satisfactory theory of best approximations. However, the problem has been studied for specific choices of $\theta$, notably by Lagarias \cite{Lagarias-1982}, Chekhova-Hubert-Messaoudi \cite{ChekhovaHubertMessaoudi-2001}, Chevallier \cite{Chevallier-2013}, and Hubert-Messaoudi \cite{HubertMessaudi-2006} in the case when the coordinates of $\theta$ lie in some cubic number fields. We follow the presentation of \cite{HubertMessaudi-2006}. 

Consider the polynomial $P(x)=x^3-ax^2-bx-1$ with integer coefficients $a,b$ satysfying ($a\geq 0$ and $0\leq b \leq a+1$) or ($a\geq 2$ and $b=-1$). Assume further that the polynomial $P$ has a unique real root $\beta>1$ and a pair of complex conjugate roots $\alpha, \bar{\alpha}$. (The condition on $a$ and $b$ arises from the work of Akiyama \cite{Akiyama-2000} who classified Pisot units of degree 3 satisfying the so-called Property (F) concerning finiteness of certain $\beta$-expansions.) Then $\beta$ generates an imaginary cubic Pisot field. We will show that the set $\{\lbr \beta^i \rbr \mid i \in \N_0\} $ is generalised polynomial. To this end, we use the results of Hubert and Messaoudi who studied the best approximations of the point $\theta=(\beta^{-1}, \beta^{-2})$ \cite[Theorem 1]{HubertMessaudi-2006}.

\begin{theorem}[\cite{HubertMessaudi-2006}]\label{bambambabam} Let the polynomial $P(x)=x^3-ax^2-bx-1$ with roots $\beta, \alpha, \bar{\alpha}$ be as above. Define the sequence $(R_i)_{i\geq 0}$ by the recurrence formula $R_0=1$, $R_1=a$, $R_2=a^2+b$, $$R_{i}=aR_{i-1}+bR_{i-2}+R_{i-3}, \quad i\geq 3.$$ Then there exists a norm $N$ on $\R^2$ (called \emph{a Rauzy norm}) such that the sequence of best approximations of the point $\theta=(\beta^{-1}, \beta^{-2})$ with respect to the norm $N$ coincides except for finitely many elements with the sequence $(R_i)$. Furthermore, there exists a sequence of real numbers $(m_q)_{q\geq 1}$ satisfying the following properties:

\begin{enumerate}
\item\label{bambambabam.mapprox} There exists a constant $c>0$ such that for all $q\geq 1$ and $p\in \Z^2$ if $N(q\theta-p)<c$, then $$N(q\theta-p)=N_0(q\theta)=m_q.$$
\item\label{bambambabam.mmin} If $q<R_i$, then $m_{R_i}<m_q$.
\item\label{bambambabam.mwzor} For $i\geq 0$ we have $m_{R_i}=m_1|\alpha|^n$. 
\item We have $\liminf_{q\to \infty} m_q =0.$
\item\label{bambambabam.beta} We have $$\beta^i=R_i + \frac{b\beta+1}{\beta^2} R_{i-1} + \frac{1}{\beta} R_{i-2}, \quad i\geq 2.$$
\item\label{bambambabam.nwzor} The norm $N$ is given by the formula $$N(x)=|(\alpha+b/\beta)x_1+x_2/\beta|, \quad x=(x_1,x_2)\in \R^2.$$
\end{enumerate}
\end{theorem}
\begin{proof} This is proven in \cite[Theorem 1, Theorem 3, Corollary 2, Proposition 7, Lemma 4, Lemma 2, and formula (3)]{HubertMessaudi-2006}. The sequence $m_q$ is denoted $N(\delta(q))$ there.
\end{proof}
\begin{remark}
The choice of the initial values of $(R_i)$ might seem unnatural. Note however that it corresponds to $R_0=1$, $R_{-1}=R_{-2}=0$.
\end{remark}

We are now ready to prove the following theorem. 

\begin{theorem}\label{LAcrime} Let $P$ be a polynomial $P(x)=x^3-ax^2-bx-1$ with integer coefficients $a,b$ satisfying ($a\geq 0$ and $0\leq b \leq a+1$) or ($a\geq 2$ and $b=-1$). Assume that $P$ has a unique real root $\beta$ and that $\beta>1$. Then the set $\{\lbr \beta^i \rbr \mid i\in \N_0\}$ is generalised polynomial.\end{theorem}

\begin{proof} By Proposition \ref{przepraszamzetakpozno}, instead of $\lbr \beta^i \rbr$, we may study the sequence $R_i$ of the previous theorem. It is easy to verify that the sequence $(R_i)$ is strictly increasing for $i\geq 6$. We will start by constructing a generalised polynomial $g \colon \Z \to \R$ that is increasing and which takes the value $g(q)=m_q^{-2}$ at points $q=R_i$ with $i$ sufficiently large. The sequence $R_i$ satisfies a third order linear recurrence with characteristic polynomial $P$ and hence $$R_i=u\beta^i+v\alpha^i+w\bar{\alpha}^i$$ for some $u,v,w\in \CC$. Hence, $\lim_{i\to \infty} (R_i-u\beta^i)=0$, and thus for $i$ large enough we have $$R_{i-1}=\lbr R_{i}/\beta\rbr, \quad R_{i-2}=\lbr  R_{i}/\beta^2\rbr.$$
Define a generalised polynomial $g$ by the formula $$g(q)= m_1^{-2}\left(q + \frac{b\beta+1}{\beta^2} \lbr q/\beta \rbr + \frac{1}{\beta}\lbr  q/\beta^2\rbr\right), \quad q\in\Z.$$
Then by Theorem \ref{bambambabam}.(\ref{bambambabam.mwzor}) and \ref{bambambabam}.(\ref{bambambabam.beta}), we have $g(R_i)=m_1^{-2}|\beta|^{i}=m_{R_i}^{-2}$ for large $i$. (Note that $|\alpha|=|\beta|^{-1/2}$.) Furthermore, $g$ is increasing. 

Consider the set $$S=\{q\geq 1 \mid N_0(q\theta)<g(q)^{-1/2}\}.$$ Since the sequence $\{R_i \mid i\in \N_0 \}$ is up to finitely many terms the sequence of best approximations of $\theta$, applying Theorem \ref{bambambabam}.(\ref{bambambabam.mapprox}) and \ref{bambambabam}.(\ref{bambambabam.mmin}) and the fact that $g$ is increasing and $g(R_i)=m_{R_i}^{-2}$ tends to $0$ as $i\to \infty$, we see that $S$ coincides with the set $\{R_i \mid i\in \N_0\}$ up to a finite set. In order to show that $S$ is generalised polynomial, it is enough to write the expression $N_0(q\theta)$ in terms of a generalised polynomial in $q$. We will do that under the additional assumption that $N_0(q\theta)$ is small. Assume that $N_0(q\theta)< \Imm(\alpha)/2$ and let $p=(p_1,p_2)\in \Z^2$ be such that $N_0(q\theta)=N(q\theta-p)$. By Theorem \ref{bambambabam}.(\ref{bambambabam.nwzor}), we have $$N(q\theta-p)=|(\alpha+b/\beta)(q\beta^{-1}-p_1)+(q\beta^{-2}-p_2)/\beta|< \Imm(\alpha)/2.$$ By looking at the imaginary part of the expression, it follows that $|q\beta^{-1}-p_1|<1/2$, and hence $p_1$ is uniquely determined as $p_1=\lbr q\beta^{-1}\rbr.$ Since $p$ minimises $N(q\theta-p)$, we also have $$p_2=\lbr \beta(\alpha+b/\beta)(q\beta^{-1}-p_1)+q\beta^{-2}\rbr.$$

Therefore we have \begin{align*}N_0(q\theta)\leq &|(\alpha+b/\beta)(q\beta^{-1}-\lbr q\beta^{-1}\rbr)+\\+&(q\beta^{-2}-\lbr \beta(\alpha+b/\beta)(q\beta^{-1}-\lbr q\beta^{-1}\rbr)+q\beta^{-2}\rbr)/\beta|\end{align*}
and  equality holds if furthermore $N_0(q\theta)< \Imm(\alpha)/2$. Call the expression on the right hand side of this inequality $h(q)$. Rewriting the  norm in $\CC$ in terms of the coordinates, we see that $h(q)^2$ is given by a generalised polynomial. Therefore the set $$T=\{q\geq 1 \mid h(q)^2<g(q)^{-1}\}$$ coincides with $S$ up to a finite set. Hence, the set $\{R_i \mid i\in \N_0\}$ is generalised polynomial. \end{proof} 
\section{Concluding remarks}

We conclude with some general remarks and questions which naturally appeared during the work on this project.

To begin with, we note that the difficulty in understanding the behaviour of arbitrary (possibly sparse) generalised polynomials should not come as a surprise. Indeed, many deep questions in number theory can be reduced to the task of deciding whether a certain generalised polynomial is identically $0$ or not. We cite one example which we find particularly striking.

\begin{conjecture*}[Littlewood]
	Let $\a,\b \in \RR$ and $\e > 0$. Then the generalised polynomial $f \colon \NN \to \{0,1\}$ given by
	$$
		f(n) = \ifbra{ n \fpa{n\a} \fpa{n \b} < \e }
	$$
	is not identically $0$.
\end{conjecture*}
It is known by the work of Einsiedler, Katok and Lindenstrauss \cite{EinsiedlerKatokLindenstrauss-2006} that the set of $(\a,\b)$ for which the above fails is either empty or of Hausdorff dimension $0$, but the conjecture remains unresolved.

In a companion paper, we study automatic sequences that are given by generalised polynomials. We reduce the problem considered there to the question 
whether the set $\set{k^t}{t \in \NN}$ is generalised polynomial; we suspect it is not. On the other hand, by Example \ref{ex:lin-rec-is-gp}, the set of Fibonacci numbers, which can be essentially written as $\set{ \lbr \varphi^t \rbr }{ t \in \NN}$, where $\varphi = (1+\sqrt{5})/2$, is generalised polynomial. Similarly, in Proposition \ref{thm:quadraticPisotisGP}, Theorem \ref{LAcrime} and Proposition \ref{przepraszamzetakpozno} we show that the sets of values of  certain classes of linear recurrence sequences of degree two and three are generalised polynomial, and the same is the case for the set $\set{ \lbr \beta^t \rbr }{ t \in \NN}$, where $\beta$ is either a quadratic Pisot unit or a certain non-totally real cubic Pisot unit. This suggests the following questions.

\begin{question}\label{Q:Pisot}
	For which $\lambda > 1$ is the set $$E_\lambda = \set{\lbr\lambda^i \rbr}{i \in \NN_0}$$ generalised polynomial?
\end{question}

We note that the set of such $\lambda$ includes at least all the Pisot units of degree $2$ and some Pisot units of degree $3$. We may state a more precise version of the previous question.

\begin{question}\label{Q:truePisot}
	Let $\beta$ be a Pisot number. Is it true that if the set  $$E_\beta = \set{\lbr\beta^i \rbr}{i \in \NN_0}$$ is generalised polynomial, then $\beta$ is a Pisot unit? Does the converse hold?
\end{question}

An intimately related question is the following one. 

\begin{question}\label{Q:linearrecurrent}
	For which linear recurrence sequences $n=(n_i)_{i\geq 0}$ with  values in $\Z$ is the set $$E_n= \set{ n_i }{i \in \N_0}$$ generalised polynomial?
\end{question}

By Proposition \ref{przepraszamzetakpozno}, Question \ref{Q:truePisot} can be regarded as a special case of Question \ref{Q:linearrecurrent}.

Looking at these examples from the perspective of recursive relations with non-constant coefficients (as in Remark \ref{lawnmower}), we may also ask the following question.

\begin{question}\label{Q:BApprox}
	Let $(n_i)_{i\geq 0}$ be a sequence given by $n_0 = 0,\ n_1 = 1$ and the recursive relation
	$$n_{i} = a_i n_{i-1} + b_i n_{i-2},$$ with $a_i,b_i \in \ZZ$ bounded in absolute value. For which $a = (a_i)_{i\geq 0},\ b = (b_i)_{i\geq 0}$ is the set 
	$$E_{a,b} = \set{n_i}{i \in \NN}$$ generalised polynomial?
\end{question}

\bibliographystyle{alpha}
\bibliography{bibliography}

\end{document}